\documentclass[reqno]{amsart} 

\usepackage{amsthm,amsmath,amsfonts,amssymb}
\usepackage[utf8]{inputenc} 
\usepackage[T1]{fontenc}
\usepackage{graphicx}
\usepackage{caption}
\usepackage{multicol}
\usepackage[hidelinks]{hyperref}
\usepackage{enumitem}

\usepackage{cleveref}

\numberwithin{equation}{section}
\theoremstyle{plain}
\newtheorem{lem}[equation]{Lemma}

\newtheorem{prop}[equation]{Proposition}
\newtheorem{thm}[equation]{Theorem}
\newtheorem{theorem}[equation]{Theorem}
\newtheorem{cor}[equation]{Corollary}
\newtheorem{conj}[equation]{Conjecture}

\theoremstyle{definition}
\newtheorem{definition}[equation]{Definition}

\newtheorem{remark}[equation]{Remark}

\newtheorem*{que*}{Question}
\newtheorem*{remark*}{Remark}
\newtheorem*{definition*}{Definition}

\newcommand{\cq}{\mathcal{Q}}
\newcommand{\od}{\widehat{\Sigma}}
\newcommand{\Si}{\Sigma}

\newcommand{\bC}{\mathfrak C}

\newcommand{\ca}{\mathcal {A}} 
\newcommand{\cb}{\mathcal {B}}

\newcommand{\fan}{\operatorname{Fan}}
 
\newcommand{\lk}{\operatorname{lk}} 
\newcommand{\ch}{\mathcal {H}}

\newcommand{\ck}{\mathcal {K}}

\newcommand{\act}{\curvearrowright}

\newcommand{\bb}[2][]{\ensuremath{\mathbb{#2}^{#1}}}

\newcommand{\PA}{P\!A}
\newcommand{\bSD}{\Delta} 
\newcommand{\sd}{\bSD}

\newcommand{\falk}{\mathbb F}

\newcommand{\pos}{\mathrm{pos}}

\renewcommand{\neg}{\mathrm{neg}}

\let\ced\c
\renewcommand{\c}{\mathcal}

\renewcommand{\epsilon}{\varepsilon}

\renewcommand{\subset}{\subseteq}

\title{A new class of affine $K(\pi,1)$  arrangements}

\author[K.~Goldman]{Katherine Goldman$^{\dag}$}
    
    \address[K.~Goldman]{
    Department of Mathematics and Statistics,
    McGill University,
    Burnside Hall,
    805 Sherbrooke Street West,
    Montreal, QC,
    H3A 0B9, Canada}
    
    \email{kat.goldman@mcgill.ca}
    
    \thanks{$\dag$ Partially supported by NSF grant DMS-2402105}

\author[J.~Huang]{Jingyin Huang$^{\ast}$}

    \address[J.~Huang]{
        Department of Mathematics,
        The Ohio State University, 
        231 W. 18th Ave,
        Columbus, OH 43210, U.S.              
    }
    
    \email{huang.929@osu.edu}
    
    \thanks{$\ast$ Partially supported by a Sloan fellowship and NSF grant DMS-2305411}

\begin{document}

\begin{abstract}
    We show that a certain class of affine hyperplane arrangements are $K(\pi,1)$ by endowing their Falk complexes with an injective metric. This gives new examples of infinite $K(\pi,1)$ arrangements in dimension $n>2$.
\end{abstract}

\maketitle

\section{Introduction}
Let $\ca$ be an \emph{affine hyperplane arrangement} in $\mathbb R^n$, i.e., a locally finite collection of affine hyperplanes in $\mathbb R^n$. We consider the complex manifold which is the complement of the following collection of hyperplanes in $\mathbb C^n$:
$$M(\ca)=\mathbb C^n-\bigcup_{H\in \ca}(H\otimes \mathbb C).$$
It is an important question to understand the topology of $M(\ca)$, see e.g.~\cite{falk1986homotopy,falk1998homotopy}. We will be specifically interested in the asphericity of $M(\ca)$. 
If the manifold $M(\ca)$ is aspherical, we call $\ca$ a \emph{$K(\pi,1)$ arrangement}. 

Unlike the situation of knot complements in $\mathbb S^3$, asphericity of $M(\ca)$ is a relatively rare phenomenon. However, there are some specific classes of $\ca$ where asphericity is known, for example:
\begin{enumerate}
    \item $\ca$ is central and simplicial by Deligne \cite{deligne};
    \item $\ca$ is supersolvable by Terao \cite{terao1986modular};
    \item $\ca$ is certain type of line arrangement in $\mathbb R^2$ by Falk \cite{falk1995k};
    \item $\ca$ is the collection of reflection hyperplanes associated with an affine Coxeter group by Paolini and Salvetti \cite{paolini2021proof}.
\end{enumerate}
These results are obtained through different means: (1) and (4) rely heavily on Garside theory; (2) is obtained through a fibration argument; (3) uses a form of conformal non-positive curvature for 2-dimensional complexes, allowing one to compute the second homotopy group directly. 
Given that there are relatively few methods and examples of aspherical arrangements when $n\ge 3$, it is desirable to extend Falk's method over dimension 2, which is the goal of this article. 
In higher dimensions, we must use a different notion of non-positive curvature in place of the conformal non-positive curvature in \cite{falk1995k} which can only be used in dimension 2.

Given an affine arrangement $\ca$, an \emph{$\ca$-vertex} is a point in $\mathbb R^n$ which can be realized as intersection of elements of $\ca$. The \emph{local arrangement} at an $\ca$-vertex $x$, denoted $\ca_x$, is the collection of all hyperplanes in $\ca$ that contain $x$. Deligne's result concerns the \emph{central} arrangements, i.e., those where $\ca = \ca_0$ with $0$ denoting the origin of $\mathbb R^n$. An interesting feature of Falk's result is the local-to-global phenomenon that for certain classes of non-central arrangements $\ca$, one can detect the asphericity of $M(\ca)$ by looking at the combinatorial features of its local arrangements. Motivated by this, we consider the following class of arrangements characterized by their local arrangements.
\begin{definition}
 \label{def:admissible1}
We say an affine hyperplane arrangement $\ca$ in $\mathbb R^n$ is \emph{admissible}, if at each $\mathcal A$-vertex $x$, the local arrangement at $x$ is a translate of the following four types:
\begin{enumerate}
    \item (type $B_n$) $x_i\pm x_j=0$ for $1\le i\neq j\le n$ and $x_i=0$ for $1\le i\le n$;
    \item (type $D_n$) $x_i\pm x_j=0$ for $1\le i\neq j\le n$;
    \item (skewed type $A_n$) $x_i=0$ for $1\le i\le n$ and $x_i=x_j$ for $1\le i\neq j\le n$, or any image of this this arrangement under the $(\bb{Z}/2\bb{Z})^n$ action on $\bb[n]{R}$ by reflections about the coordinate hyperplanes; or
    \item a product of the previous types.
\end{enumerate}
\end{definition}

Note that any affine Coxeter arrangement associated with a non-exceptional affine Coxeter group (i.e., types $\widetilde A_n,\widetilde B_n,\widetilde C_n,\widetilde D_n$) is an admissible arrangement\footnote{We use a different description of the $\widetilde A_n$ arrangement, where the hyperplanes are $x_i\in \mathbb Z$ for $1\le i\le n$ and $x_i-x_j\in \mathbb Z$ for $1\le i\neq j\le n$. This does not affect the topology of $M(\ca)$.}. However most of the arrangements in Definition~\ref{def:admissible1} are not Coxeter arrangements.

\begin{theorem}
\label{thm:main1}
Let $\ca$ be an admissible affine arrangement in $\mathbb R^n$ which is invariant under the action of a discrete translation subgroup $\mathbb Z^n$ of $\mathbb R^n$ (this does not have to be the usual embedding of $\mathbb Z^n$). Suppose $n\le 4$. Then $\ca$ is a $K(\pi,1)$ arrangement. More generally, modulo a group theoretical conjecture on the spherical Artin group of type $D_n$ (Conjecture~\ref{conj:dn}), $\ca$ is a $K(\pi,1)$ arrangement for any $n$.
\end{theorem}

In the situation of the above theorem, we have a free action of $\mathbb Z^n$ on $M(\ca)$. Then the fundamental group of $M(\ca)/\mathbb Z^n$ can be viewed as a generalization of the affine Artin groups. When $\ca$ is an affine Coxeter arrangement, this gives a finite index subgroup of the corresponding affine Artin group. 

\begin{cor}
Under the assumption of Theorem~\ref{thm:main1}, the manifold $M(\ca)/\mathbb Z^n$ is homotopy equivalent to a finite aspherical cell complex, which is the quotient of the Salvetti complex of $\ca$ (\cite{s87}) by a free action of $\mathbb Z^n$. In particular, the fundamental group of $M(\ca)/\mathbb Z^n$ is of type $F$.
\end{cor}

Theorem \ref{thm:main1} is a special case of a more general statement which does not require $\ca$ to be $\mathbb Z^n$-invariant, but rather requires it have \emph{finite shape}, meaning there are finitely many isometry types of chambers of the arrangement. See Theorem \ref{thm:main} for the full statement.

Now we give more concrete examples of arrangements where Theorem~\ref{thm:main1} (or more specifically, Theorem \ref{thm:main}) applies. For each dimension $n$, we will construct an infinite family of finite affine arrangements $\ch_{k,n}$, and an infinite family of $\mathbb Z^n$-invariant affine arrangements $\ck_{k,n}$. These arrangements are not Coxeter arrangements, and their complexified complement do not admit iterated fibration structure in an obvious way. The finite arrangements are not simplicial. So the $K(\pi,1)$ results for these arrangements are new. 

\begin{definition}
For $k\ge 1$, let $\ch_{k,n}$ be the affine hyperplane arrangement in $\mathbb R^n$ given by $x_i \in \{-2k-1, -2k+1, \dots,-3, -1, 1, 3, \dots, 2k - 1, 2k + 1\}$ for $1 \le i \le n$ and $x_i \pm x_j = 0$ for $1\le i\neq j\le n$. This family $\ch_{k,n}$ generalizes Falk's \cite[Example 3.13]{falk1995k} which is neither supersolvable or simplicial.

For $k\ge 1$, let $\ck_{k,n}$ be the affine hyperplane arrangement in $\mathbb R^n$ given by $x_i\in \mathbb Z$ for $1\le i\le n$, and $x_i+x_j \in 2k\mathbb Z+1$, $x_i-x_j \in 2k\mathbb Z$ for $1\le i\neq j\le n$.
\end{definition}

\begin{thm}(=Theorem~\ref{thm:examples})
For $n\le 4$ and any $k\ge 1$, the arrangements $\ch_{k,n}$ and $\ck_{k,n}$ are $K(\pi,1)$ arrangements. More generally, modulo a group theoretical conjecture on the spherical Artin group of type $D_n$ (Conjecture~\ref{conj:dn}), $\ch_{k,n}$ and $\ck_{k,n}$ are $K(\pi,1)$ arrangements for any $n,k$.
\end{thm}

For $n=2$, Falk \cite{falk1995k} constructed a (locally infinite) 2-dimensional complex $\falk_\ca$ that is homotopy equivalent to the universal cover of $M(\ca)$. This construction of Falk can be generalized to higher dimensions without too much difficulty \cite{huang2024cycles}. Later, we show $\falk_\ca$ is homotopy equivalent to the universal cover of $M(\ca)$ whenever all the local arrangements are $K(\pi,1)$. Thus all the above results rely on showing the complex $\falk_\ca$ is contractible. 

The proof of contractibility relies on constructing certain types of metric on $\falk_\ca$. It turns out that there is a natural projection from $\falk_\ca$ to a subset of $\mathbb R^n$ that is injective on each face. We define a metric
 \(d_\infty\) on $\falk_\ca$ which is obtained by equipping each face of the Falk complex, via its projection to a \(\mathbb R^n\), with the ambient \(\ell^\infty\)-metric and then taking the induced path metric; see Definition~\ref{def:falk metric} for details.

Recall that a geodesic metric space $X$ is \emph{injective} if any pairwise intersecting closed metric balls in $X$ have non-empty common intersection. For example, $\mathbb R^n$ equipped with the $\ell^\infty$ metric is injective. Injective metric spaces are contractible, and they are connected to the above theorems in the following way.

\begin{thm}
\textup{(= Theorem \ref{thm:main} and Theorem~\ref{thm:examples})}
Under the assumptions of any of the previous theorems, the Falk complex $\falk_\ca$ with the metric $d_\infty$ is an injective metric space.
\end{thm}

The reason for us to consider skewed type $A_n$ arrangements as local arrangements, rather than the standard $A_n$ arrangements, is for the compatibility of arranging injective metrics on $\falk_\ca$. A substantial part of the article (Section~\ref{sec:check links}) is devoted to checking such compatibility (Proposition~\ref{prop:key} and Remark~\ref{rmk:contrast}). More specifically, we show that a natural subdivision of the spherical Deligne complex associated with the skewed type $A_n$ arrangement satisfies a combinatorial criterion of Haettel~\cite{haettel2021lattices}. This criterion can be used to show the metric $d_\infty$ on $\falk_\ca $ is injective, since these spherical Deligne complexes arise as links of vertices in the Falk complexes. 
In contrast, the spherical Deligne complex associated with the usual $A_n$ arrangement is known not to satisfy this criterion.

It is a topic of independent interest to produce natural examples of injective metric spaces arising from group theory, and the above theorem gives many such examples. As the above theorem becomes conditional for $n\ge 5$, we also have the following variation that is unconditional for all dimensions, where we allow the local arrangements to be a mixture of $A_n$-type and $B_n$-type.
The following corollary is more interesting in terms of providing new examples of injective metric spaces, rather than $K(\pi,1)$ results, as one can prove the arrangements in the following corollary are $K(\pi,1)$ via an iterated fibration argument.

\begin{cor} \textup{(= Corollary~\ref{cor:AB})}
\label{cor:AB intro}
Suppose $\ca$ is a complete, finite shape, affine arrangement in $\mathbb R^n$ such that for each $\mathcal A$-vertex $x$, the local arrangement at $x$ is a translate of one of the following three types:
\begin{enumerate}
    \item (type $B_n$) $x_i\pm x_j=0$ for $1\le i\neq j\le n$ and $x_i=0$ for $1\le i\le n$;
    \item (skewed type $A_n$) $x_i=0$ for $1\le i\le n$ and $x_i=x_j$ for $1\le i\neq j\le n$, or any image of this this arrangement under the $(\bb{Z}/2\bb{Z})^n$ action on $\bb[n]{R}$ by reflections about the coordinate hyperplanes; or
    \item a product of the previous types.
\end{enumerate}
Then $(\falk_\ca,d_\infty)$ is an injective metric space and $\ca$ is a $K(\pi,1)$ arrangement.
\end{cor}

\subsection*{Structure of the article} In Section~\ref{sec:prelim} we collect some background material. In Section~\ref{subsec:deligne complex} we define Falk complexes of affine hyperplane arrangements and prove some of them admit injective metrics, modulo a key proposition (Proposition~\ref{prop:key}) about skewed $A_n$ arrangements. Section~\ref{sec:check links} is devoted to this key proposition. Section~\ref{sec:example} contains some new, concrete examples of affine arrangements where our results apply.

\subsection*{Acknowledgements}
We are grateful to the anonymous referee for a careful reading of the manuscript and for many helpful comments and suggestions.
\section{Preliminaries}
\label{sec:prelim}
\subsection{Hyperplane arrangements and their dual polyhedra}\label{subsec:zonotope}
An \emph{affine hyperplane arrangement} in the vector space $\mathbb R^n$ is a locally finite family $\mathcal A$ of affine hyperplanes.
 Let $\cq(\ca)$ be the set of nonempty affine subspaces that are intersections of subfamilies of $\ca$ (here $\mathbb R^n\in\cq(\ca)$ as the intersection of an empty family). Each point $x\in\mathbb R^n$ belongs to a unique element of $\cq(\ca)$ that is smallest with respect to inclusion, called the \emph{support} of~$x$. A \emph{fan} of $\ca$ is a maximal connected subset of $\mathbb R^n$ consisting of points with the same support. 
Denote the collection of all fans of $\ca$ by $\fan(\ca)$.
Note that $\mathbb R^n$ is the (disjoint) union of $\fan(\ca)$. 
We define a partial order on $\fan(\ca)$ so that $U_1<U_2$ if $U_1$ is contained in the closure of $U_2$.
Let $b\Si_{\ca}$ be the simplicial complex that is the geometric realization of this poset.
For each $U\in \fan(\ca)$, we choose a point $x_U\in U$. 
This gives a piecewise linear embedding $b\Si_{\ca}\subset\mathbb R^n$ sending the vertex of $b\Si_\ca$ corresponding to $U$ to $x_U$.

By \cite[pp.~606-607]{s87}, the simplicial complex $b\Si_{\ca}$ is the barycentric subdivision of a combinatorial complex $\Si_{\ca}$ whose vertices correspond to the top-dimensional fans.
Namely, for each vertex of $b\Si_{\ca}$ corresponding to $U\in \fan(\ca)$, the union of all the simplices of $b\Si_{\ca}$ corresponding to chains with smallest element $U$ is homeomorphic to a closed disc \cite[Lem 6]{s87}, which becomes the face of $\Si_{\ca}$ corresponding to~$U$. We will  sometimes view $b\Si_\ca$ and $\Si_\ca$ as subspaces of $\mathbb R^n$.
For $B\in \cq(\ca)$, a face $F$ of $\Si_{\ca}$ is \emph{dual} to $B$,  if %
$B$ contains the fan $U$ corresponding to $F$ and $\dim(B)=\dim(U)$.
We equip the 1-skeleton of $\Si_{\ca}$ with the path metric $d$ such that each edge has length $1$. Given vertices $x,y\in \Si^0_{\ca}$, it turns out that $d(x,y)$  is the number of hyperplanes separating $x$ and~$y$ \cite[Lem~1.3]{deligne}.  

\begin{lem}[{\cite[Lem 3]{s87}}]
	\label{lem:gate}
	Let $x\in \Si^0_{\ca}$ and let $F$ be a face of $\Si_{\ca}$. Then there exists unique $\Pi_F(x)\in F^0$ such that $d(x,\Pi_F(x))\le d(x,y)$ for any $y\in F^0$. \end{lem}

The vertex $\Pi_F(x)$ is called the \emph{projection} of $x$ to $F$.  A hyperplane $H\in \ca$ \emph{crosses} a face $F$ of $\Si_{\ca}$ if $H$ is dual to an edge of $F$. For an edge $xy$ of $\Si_\ca$, if the hyperplane dual to $xy$ crosses $F$, then $\Pi_F(x)\Pi_F(y)$ is an edge dual to the same hyperplane, otherwise we have $\Pi_F(x)=\Pi_F(y)$. Thus $\Pi_F$ extends naturally to a map $\Si^1_\ca\to F^1$.

\begin{lem}(\cite[Lem 3.2]{huang2025353})
\label{lem:unlabelled}
	Let $E$ and $F$ be faces of $\Si_\ca$. Then $\Pi_F(E^0)=F'^0$ for some face $F'\subset F$. 
\end{lem}

In the situation of Lemma~\ref{lem:unlabelled}, we write $F'=\Pi_F(E)$.

The assignment $E\to \Pi_F(E)$ gives rise to a piecewise linear map $\Pi_F\colon\Si_\ca \cong b\Si_\ca\to bF\cong F$.

\subsection{The Salvetti complex}
\label{sec:Salvetti}
Let $V=\Si_\ca^0$. 
Consider the set of pairs $(F,v)$, where $F$ is a face of $\Si_\ca$ and $v\in V$. We define an equivalence relation $\sim$ on this set by $(F,v)\sim (F',v')$ whenever  $F=F'$ and $\Pi_F(v') = \Pi_F(v).$
Note that each equivalence class $[F,v']$ contains a unique representative of form $(F,v)$ with $v\in F^0$.  The \emph{Salvetti complex} $\widehat\Si_\ca$ \cite[p.~608]{s87} is defined as follows. For each face $F\subset \Si_\ca$ and each equivalent class $[F,v]$, we associate a face of $\widehat \Si_\ca$, which is a copy of $F$. So we will also use $[F,v]$ to denote a face of $\widehat\Si_\ca$. Given two faces $F\subset E$ of $\Si_\ca$, we identify $[F,v]$ as a sub-face of $[E,w]$ if there exists $u\in V$ such that $[F,u]=[F,v]$ and $[E,u]=[E,w]$.

For example, for each edge $F=v_0v_1$ of $\Si_\ca$, we obtain two edges $[F, v_0]$ and $[F, v_1]$ of $\widehat\Si_\ca$, glued along their endpoints $[v_0,v_0]$ and $[v_1,v_1]$. We orient the edge $[F, v_0]$ from $[v_0,v_0]$ to $[v_1,v_0]=[v_1,v_1]$. Then $\widehat\Si_\ca^0=V$, while $\widehat\Si_\ca^1$ is obtained from $\Si_\ca^1$ by doubling each edge.  
Each edge of $\widehat\Si_\ca$ of the form $[F, v]$ is oriented so that its endpoint is farther from $v$ in~$V=\widehat\Si^0_\ca$ than its starting point.

There is a natural cellular map $p\colon\widehat\Si_\ca\to\Si_\ca$ by sending $[F,v]$ to $F$.
For each subcomplex $Y$ of $\Si_\ca$, we write $\widehat Y=p^{-1}(Y)$. 
If $F$ is a face of $\Si_\ca$, then $\widehat F$ is a \emph{standard subcomplex} of $\widehat\Si_\ca$. %

\begin{lem}%
	\label{lem:compactible}
	Let $E$ and $F$ be faces of $\Si_\ca$.
	If $[E,v_1]=[E,v_2]$ as faces of $\widehat \Si_\ca$, then $[\Pi_F(E),v_1]=[\Pi_F(E),v_2]$ as faces of $\widehat \Si_\ca$.
\end{lem}

\begin{definition}%
	\label{def:retraction}
	Let $F$ be a face of $\Si_\ca$. Following \cite[Thm 2.2]{godelle2012k}, we define a cellular map $\Pi_{\widehat F}\colon \widehat\Si_\ca\to \widehat F$ sending a face $[E,v]$ of $\widehat\Si_\ca$ to the face $[\Pi_F(E),v]$. This map is well-defined by Lemma~\ref{lem:compactible}, and it is a retraction.
\end{definition}

The following key property of $\Pi_{\widehat F}$ follows directly from Definition \ref{def:retraction}.

\begin{lem}	\label{lem:retraction property}
	Let $E$ and $F$ be faces of $\Si_\ca$. Then $\Pi_{\widehat F}(\widehat E)=\widehat{\Pi_F(E)}$.
\end{lem}

Let $\mathcal A\otimes \mathbb C$ be the complexification of $\mathcal A$, which is a collection of affine complex hyperplanes in $\mathbb C^n$. Define
$$M(\mathcal A\otimes \mathbb C)=\mathbb C^n - \bigcup_{H\in\mathcal A} (H\otimes \mathbb C).$$
It follows from \cite[Thm 1]{s87} that $\widehat\Si_\ca$ is homotopy equivalent to $M(\mathcal A\otimes \mathbb C)$, and so they have isomorphic fundamental groups. We say $\mathcal A$ is a \emph{$K(\pi,1)$ arrangement} if $M(\ca\otimes \mathbb C)$ is aspherical, or equivalently, the Salvetti complex $\od_\ca$ is aspherical.

\subsection{Injective metric spaces}
\label{subsec:injective}
A \emph{geodesic (segment)} in a metric space $X$ is the image of an isometric embedding from a closed interval in $\mathbb R$ (equipped with the usual metric on $\mathbb R$) to $X$. A metric space is \emph{geodesic}, if every pair of points are joined by a geodesic in the space. A \emph{ball} $B(x,r)$ in a metric space is the collection of points of distance $\le r$ from a point $x$, which is the center of the ball.
A geodesic metric space is \emph{injective} if every collection of balls in the metric space which have non-empty pairwise intersection also have non-empty common intersection. As an example, the $\ell^\infty$-norm on $\mathbb R^n$ induces an injective metric on $\mathbb R^n$. It is known that any injective metric space is contractible, as it is possible to select a geodesic joining each pair of points such that this geodesic varies continuously depending on the endpoints (see e.g. \cite[Prop 3.8]{lang2013injective}). Note the each ball in an injective metric space $X$, endowed with the induced metric from $X$, is itself an injective metric space (the ball is geodesic by \cite[Prop 3.8 (1)]{lang2013injective}).

A metric space is \emph{uniformly locally injective} if there is an $\epsilon>0$ such that each ball of radius $\epsilon$ is injective. %

We recall the following local-to-global criterion for injective metric space.

\begin{theorem}\cite[Thm 1.14]{haettel2021lattices}
    \label{thm:local-to-global}
Let $X$ be a geodesic metric space that is complete,
simply connected and uniformly locally injective. Then $X$ is injective.
\end{theorem}

\begin{remark}
Groups admitting interesting actions on injective metric spaces often enjoy strong group-theoretic properties. In geometric group theory, one commonly considers groups acting \emph{geometrically} on injective metric spaces, meaning that the action is properly discontinuous, cocompact (i.e.\ the quotient is compact), and by isometries. For example, fundamental groups of the complements of complexified finite central simplicial hyperplane arrangements, braid groups of well-generated complex reflection groups, and all Garside groups of finite type admit geometric actions on injective metric spaces \cite{huang2021helly}. Any group acting geometrically on an injective metric space has solvable word and conjugacy problems, admits an EZ-boundary and a Tits boundary, satisfies the Farrell--Jones conjecture with finite wreath products and the coarse Baum--Connes conjecture, and has Euclidean higher-dimensional Dehn functions; see \cite{huang2021helly} for a summary of how these properties follow from results in the literature.

In this article, we construct isometric actions of fundamental groups of certain complex hyperplane arrangement complements on injective metric spaces. The actions considered here are not properly discontinuous, and hence are not geometric, so the properties listed above do not follow automatically. Nevertheless, constructing such actions is useful because of their connection with the $K(\pi,1)$ problem, which we exploit later in the article. There are also indications that non-proper actions on injective spaces may carry further group-theoretic information. For example, Haettel \cite[Section~8]{haettel2021lattices} constructed non-proper actions of Artin groups of type $\widetilde C_n$ on injective metric spaces and used them to obtain several structural results, including descriptions of intersections of parabolic subgroups, commensurators and normalizers of parabolic subgroups, and fixed-point subgroups. Although these results are proved there only for Artin groups of type $\widetilde C_n$, it is conceivable that similar arguments could apply to some of the arrangement groups considered here once suitable actions on injective metric spaces are available. We do not pursue this direction in the present article.
\end{remark}

Now we discuss a combinatorial criterion for justifying that a metric space is injective.
We start by recalling some terminologies about posets. See \cite{davey2002introduction} for more background.

Let $P$ be a poset (i.e., a partially ordered set). For $a,b\in P$ with $a\le b$, the \emph{interval}  between $a$ and $b$, denoted by $[a,b]$, is the collection of all elements $x$ of $P$ such that $a\le x$ and $x\le b$. 
Let $S\subset P$. An \emph{upper bound} (resp. \emph{lower bound}) for $S$ is an element $x\in P$ such that $s\le x$ (resp. $s\ge x$) for every $s\in S$. The \emph{join} of $S$ is an upper bound $x$ of $S$ such that $x\le y$ for any other upper bound $y$ of $S$. The \emph{meet} of $S$ is a lower bound of $x$ of $S$ such that $x\ge y$ for any other lower bound $y$ of $S$. We will write $x\vee y$ for the join of two elements $x$ and $y$, and $x\wedge y$ for the meet of two elements (if the join or the meet exists). A poset is \emph{bounded} if it has a largest element and a smallest element. $P$ is a \emph{lattice} if $P$ is a poset and any two elements in $P$ have a join and have a meet. Later in the article, for a poset that does not have a largest element or a smallest element, sometimes we will artificially add such elements.

A chain in $P$ is a totally (or ``linearly'') ordered
subset, and a maximal chain is one that is not a proper subset of any other chain.  In any poset every chain is contained in a maximal chain, by a Zorn’s Lemma argument.  A 
poset $P$ has \emph{rank $n$} if the longest chain in $P$ has length $n$, and a finite rank poset $P$ is \emph{graded} if any two
maximal chains in $P$ have the same length; it follows from that every interval $[a,b]$ in $P $ is also graded.

The \emph{geometric realization of $P$}, denoted by $|P|$, is a simplicial complex whose vertex set is $P$, where a collection of vertices span a simplex if and only if they form a chain in $P$.

\begin{definition}
	Let $P$ be a poset. 
 We say that $P$ is \emph{bowtie free} if for any subset $\{x_1,x_2,y_1,y_2\}\subset P$ made of mutually distinct elements with $x_i<y_j$ for $i,j\in\{1,2\}$, there exists $z\in P$ such that $x_i\le z\le y_j$ for any $i,j\in\{1,2\}$.
\end{definition}

\begin{lem}\cite[Proposition 1.5]{brady2010braids}
	\label{lem:posets}
	If $P$ is a bowtie free graded poset, then any pair of elements in $P$ with a lower bound have a join, and any pair of elements in $P$ with a upper bound have a meet.
	
	Let $P$ be a bounded graded poset. Then $P$ is lattice if and only if it is bowtie free. 
\end{lem}

\begin{definition}
	\label{def:flag}
	A poset $P$ is \emph{upward flag} if any three pairwise upper bounded elements have an upper bound. A poset is \emph{downward flag} if any three pairwise lower bounded elements have a lower bound. A poset is \emph{flag} if it is both upward flag and downward flag.
\end{definition}

An $n$-dimensional \emph{unit orthoscheme} of $\mathbb R^n$ is the convex hull of $$v_0=(0,0,\ldots,0),v_1=(1,0,\ldots,0),v_2=(1,1,\ldots,0),\cdots,v_n=(1,1,\ldots,1).$$ We endow the unit orthoscheme with the $\ell^\infty$-metric. 
Following \cite[Sec 5]{brady2010braids} and \cite[Sec 1]{haettel2021lattices}, the \emph{$\ell^\infty$-orthoscheme complex} of a poset $P$, denoted by $|P|_{\infty}$, is $|P|$ endowed with the metric such that each is assigns every top dimensional simplex in $|P|$ (i.e. those corresponding to maximal chains $x_0<x_1<\cdots<x_n$) the $\ell^\infty$-metric of a
unit orthoscheme with $x_i$ corresponding to $v_i$.

Via a standard procedure (\cite[p.~65]{BridsonHaefliger1999}), the $\ell^\infty$-metric on the simplices of $|P|$ induces a \emph{pseudometric} $d$ on $|P|_\infty$, which we describe here for the convenience of the reader (a pseudometric satisfies all properties of a distance function of a metric space, except the distance of two points are allowed to be zero). Given $x,y\in |P|$, an \emph{$m$-string} from $x$ to $y$ is a sequence $(x_0=x,x_1,x_2,\ldots,x_m=y)$ such that for each $i$, $x_i$ and $x_{i+1}$ are contained in a common simplex $\sigma_i$. The \emph{length} of this string is defined to be $\sum_{i=0}^{m-1}d_{\sigma_i}(x_i,x_{i+1})$, where $d_{\sigma_i}$ denotes the metric on $\sigma_i$. Then $d(x,y)$ is defined to be infimum of the lengths of all strings from $x$ to $y$.
In the case $P$ is a graded poset with finite rank, there are only finitely many isometry types of simplices appearing in $|P|_\infty$, hence this pseudometric is a metric, and it is geodesic -- this follows from the arguments in \cite[p.~101, Corollary I.7.10]{BridsonHaefliger1999}.

\begin{theorem}(\cite[Thm 6.3]{haettel2021lattices})
\label{thm:injective criterion}
Let $P$ be a graded poset with a smallest element and with finite rank such that $P$ is bowtie free and flag. Then $(|P|_\infty,d)$ is an injective metric space.
\end{theorem}

Note that in the above theorem, due to the assumption of the existence of a smallest element in $P$, $P$ is automatically downward flag. So the theorem holds true if we replace the word ``flag'' by ``upward flag''. This is how we will use this theorem later in the article.

\subsection{Coxeter complexes and Artin complexes}
\label{subsec:coxeter and artin}
 A \emph{Coxeter diagram} $\Lambda$ is a finite simplicial graph with vertex set $S=\{s_i\}_{i \in I}$ and labels $m_{ij} \in \{3,4,\ldots ,\infty\}$ for each edge $s_is_j$. If $s_is_j$ is not an edge, we define $m_{ij}=2$. The Artin group $A_\Lambda$ is the group with generating set $S$ and relations $s_is_js_i\cdots=s_js_is_j\cdots$, with both sides alternating words of length $m_{ij}$, whenever $m_{ij}<\infty$. The Coxeter group $W_\Lambda$ is obtained from $A_\Lambda$ by adding the relations $s_i^2=1$ for each $i$. The kernel $\PA_\Lambda$ of the natural homomorphism $A_\Lambda\to W_\Lambda$ is called the \emph{pure Artin group}. 

 \begin{definition}
\label{def:Artin0}
Let $A_\Lambda$ be an Artin group with Coxeter graph $\Lambda$ and generating set $S$. Its \emph{Artin complex} $\Delta_\Lambda$ \cite{CharneyDavis,godelle2012k,cumplido2020parabolic} is a simplicial complex defined as follows. For each $s\in S$, let $A_{\hat s}$ be the standard parabolic subgroup generated by $\hat s=S\setminus\{s\}$. The vertices of $\Delta_\Lambda$ correspond to the left cosets of $\{A_{\hat s}\}_{s\in S}$. A collection of vertices span a simplex if and only if the corresponding cosets have non-empty common intersection. If $v$ is a vertex of $\Delta_\Lambda$ which is a left coset of $A_{\hat s}$, we say $v$ is \emph{type $\hat s$}. 

The \emph{Coxeter complex}~$\bC_\Lambda$ for the Coxeter group $W_\Lambda$ is defined analogously, where we replace~$A_{\hat s}$ by $W_{\hat s}<W_\Lambda$
generated by~$\hat s$. If $v$ is a vertex of $\bC_\Lambda$ which is a left coset of $W_{\hat s}$, we say $v$ is \emph{type} $\hat s$. 
\end{definition} 

\begin{lem}(\cite[Lem 6]{cumplido2020parabolic})
    \label{lem:link}
Let $\delta\subset \Delta_\Lambda$ be a simplex whose types of vertices are $\{\hat s_i\}_{i=1}^k$. Then the link $\lk(\delta,\Delta_\Lambda)$ of $\delta$ in $\Delta_\Lambda$ is isomorphic to $\Delta_{\Lambda'}$ where $\Lambda'$ is the induced subdiagram of $\Lambda$ spanned by vertices in $\Lambda\setminus\{s_1,s_2,\ldots,s_k\}$.
\end{lem}

Now consider the special case $W_\Lambda$ is a finite Coxeter group with its canonical representation $\rho
\colon W_\Lambda\to \mathbf{GL}(n,\mathbb R)$ \cite[Chap~6.12]{dbook}. A \emph{reflection} of $W_\Lambda$ is a conjugate of $s\in S$. Each reflection pointwise fixes a hyperplane in~$\mathbb R^n$, which we call a \emph{reflection hyperplane}. Let $\mathcal A$ be the family of all reflection hyperplanes. The hyperplane arrangement $\ca$ is the \emph{reflection arrangement} associated with $W_\Lambda$. We denote $\Si_\Lambda=\Si_\ca$ and $\od_\Lambda=\od_\ca$. Since $W_\Lambda$ permutes the elements of~$\mathcal A$, there is an induced action $W_\Lambda\act M(\ca\otimes \mathbb C)$ and an induced action $W_\Lambda\act \widehat\Si_\ca$, which are free.
The union of $\ca$ cuts the unit sphere of $\mathbb R^n$ into a simplicial complex, which is isomorphic to the Coxeter complex $\bC_\Lambda$ and dual to $\Si_\Lambda$.  The following are standard \cite[\S 3.2 and 3.3]{paris2014k}. 

\begin{theorem}
\label{thm:reflection arrangement}
Suppose $\ca$ is the reflection arrangement associated with a finite Coxeter group $W_\Lambda$. Then 
\begin{itemize}
	\item $\pi_1M(\ca\otimes \mathbb C)=\PA_\Lambda$ \cite{lek},
	\item $\pi_1(M(\ca\otimes \mathbb C)/ W_\Lambda)=\pi_1(\widehat\Si_\Lambda/W_\Lambda)=A_\Lambda$,
	\item $\widehat\Si^2_\Lambda/W_\Lambda$ is isomorphic to the presentation complex of $A_\Lambda$.
\end{itemize}
\end{theorem}

Now we consider the special case when $\Lambda$ is a linear graph with consecutive vertices $\{s_i\}_{i=1}^n$. We define a relation on the vertex set of $\Delta_\Lambda$ as follows. Let $x,y\in \Delta_\Lambda$ be vertices of type $\hat s_i$ and $\hat s_j$, respectively. We say $x<y$ if they are adjacent in $\Delta_{\Lambda}$ and $i<j$. A similar relation can be defined for the vertex set of the Coxeter complex $\bC_\Lambda$. This relation is actually transitive (see, e.g., \cite[Cor 6.5]{huang2023labeled}), so we obtain a poset, which is graded and has rank $n$. 

Recall that $\Lambda$ is of type $A_n$ if each edge is labeled by $3$, in which case $A_\Lambda$ is the braid group on $n+1$ stands, and $W_\Lambda$ is the symmetric group on $n+1$ letters. The corresponding reflection arrangement is obtained by intersecting the hyperplanes $\{x_i=x_j\}_{1\le i < j\le n+1}$ of $\mathbb R^{n+1}$ with the subspace
$$
V=\{(x_1,\ldots,x_{n+1})\mid \sum_{i=1}^{n+1}x_i=0\}\cong \mathbb R^n.
$$
Up to a linear transformation, we can represent an $A_n$-type reflection arrangement as the hyperplanes $\{x_i=0\}_{i=1}^n$ and $\{x_i=x_j\}_{1\le i< j\le n}$ in $\mathbb R^n$. We will call this arrangement a \emph{skewed $A_n$ arrangement}, as it differs with the actual $A_n$ arrangement by a linear transformation.

\begin{thm}
\label{thm:bowtie free An}
Suppose $\Lambda$ is of type $A_n$. Then $(\Delta^0_\Lambda,<)$ and $(\bC^0_\Lambda,<)$ are bowtie free posets.
\end{thm}
The Coxeter case follows from Tits's work \cite{tits1974buildings}, see \cite[Thm 2.10]{hirai2020uniform} for an explanation. The Artin case is due to Crisp-McCammond (unpublished), see \cite[\S 5]{haettel2021lattices} for an explanation. 

A Coxeter diagram $\Lambda$ is of type $B_n$ if $\Lambda$ is linear graph with consecutive vertices $\{s_i\}_{i=1}^n$ such that all edges are labeled by $3$ except the edge between $s_{n-1}$ and $s_n$, which is labeled by $4$. Then $W_\Lambda$ is the finite Coxeter group of type $B_n$ (sometimes called the ``signed symmetric group''), with the associated reflection arrangement made up of $\{x_i=\pm x_j\}_{1\le i< j\le n}$ and $\{x_i=0\}_{1\le i\le n}$ in $\mathbb R^n$. Instead of describing the Coxeter complex $\bC_\Lambda$ as the subdivision of the unit sphere of $\mathbb R^n$ by these hyperplanes, it is instructive to think of  $\bC_\Lambda$ as the barycentric subdivision of the boundary of the cube $[-1,1]^n$ in $\mathbb R^n$. Then vertices of $\bC_\Lambda$ are in 1-1 correspondence with points in $\mathbb R^n$ whose each coordinate belongs to $\{-1,0,1\}$ with at least one nonzero coordinate. Such a point corresponds to a vertex of type $\hat s_i$ if and only if it has $i$ nonzero coordinates. Moreover, if $P$ is obtained by adding an extra smallest element to $(\bC^0_\Lambda,<)$, then the geometric realization $|P|$ is isomorphic to the barycentric subdivision of the entire cube $[-1,1]^n$, which is naturally made of unit orthoschemes, and $|P|_\infty$ is isometric to $[-1,1]^n$ with the $\ell^\infty$ metric.

\begin{theorem}
\label{thm:Bnflag}
Suppose $\Lambda$ is of type $B_n$. Then $(\Delta^0_\Lambda,<)$ and $(\bC^0_\Lambda,<)$ are bowtie free and upward flag posets.
\end{theorem}
The Coxeter case of the above theorem is straight forward. The Artin case is proved in \cite[Prop 6.6]{haettel2021lattices}, which is deduced from Theorem~\ref{thm:bowtie free An}.

The Coxeter diagram $\Lambda$ of type $D_n$ (for $n\ge 3$) is shown in Figure~\ref{fig:dn}, where all edges are labeled by $3$. The associated reflection arrangement is $\{x_i=\pm x_j\}_{1\le i < j\le n}$ in $\mathbb R^n$.
We subdivide each edge of $\Delta_\Lambda$ connecting a vertex of type $\hat s_n$ and a vertex of type $\hat s_{n-1}$, and declare the middle point of such edge is of type $m$. Cut each top dimensional simplex in $\Delta_\Lambda$ into two simplices along the codimension 1 simplex spanned by vertices of type $m$ and $\{\hat s_i\}_{i=1}^{n-2}$. This gives a new simplicial complex, denoted by $\Delta'_\Lambda$. Define a map $t$ from $(\Delta'_\Lambda)^0$ to $\{1,2,\ldots,n\}$ by sending vertices of type $\hat s_i$ to $i$ for $1\le i\le n-2$, vertices of type $m$ to $n-1$, and vertices of type $\hat s_n$ and $\hat s_{n-1}$ to $n$. We define a relation $<$ on $(\Delta'_\Lambda)^0$ as follows. For two vertices $x,y$ of $\Delta'_\Lambda$, $x<y$ if $x$ and $y$ are adjacent in $\Delta'_\Lambda$ and $t(x)<t(y)$. The simplicial complex $\Delta'_{\Lambda}$, together with this relation, is called the \emph{$(s_n,s_{n-1})$-subdivision of $\Delta_\Lambda$}. Similarly, we can define $(s_n,s_{n-1})$-subdivision $\bC'_\Lambda$ of $\bC_\Lambda$. Note that $\bC'_\Lambda$ is isomorphic to the Coxeter complex of type $B_n$ via a type-preserving isomorphism. While $\Delta'_\Lambda$ is not isomorphic to the Artin complex of type $B_n$, the following conjecture, due to Haettel, predicts that these two complexes share the following property -- his original motivation for this conjecture is that it leads an alternative proof that Artin group of type $\widetilde D_n$ satisfies the $K(\pi,1)$-conjecture. 

\begin{conj}(Haettel)
\label{conj:dn}
Suppose $\Lambda$ is of type $D_n$ for $n\geq 3$. Then $((\Delta'_\Lambda)^0,<)$ is a poset that is bowtie free and upward flag.
\end{conj}
The poset part of the conjecture is straightforward. The bowtie free part is a consequence of \cite[Cor 8.2]{huang2023labeled}. The upward flag part reduces to checking properties of certain 6-cycles in the Artin complex $\Delta_\Lambda$, see e.g. \cite[Lem 7.6]{huang2024Dn}.
The following is proved in \cite{huang2024Dn}.
\begin{theorem}
\label{thm:dn dim 3 and 4}
Conjecture~\ref{conj:dn} holds when $n=3,4$.
\end{theorem}
\begin{figure}
    \centering
    \includegraphics[width=0.7\linewidth]{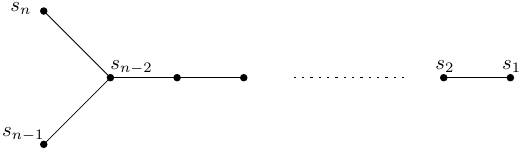}
    \caption{Coxeter diagram of type $D_n$.}
    \label{fig:dn}
\end{figure}

Due to the many notations already introduced in section, and will be introduced in the next section, we end this section with a table summarizing some of the key notations.

\begin{table}[t]
\centering
\small
\renewcommand{\arraystretch}{1.15}
\begin{tabular}{@{}p{0.15\textwidth}p{0.6\textwidth}p{0.16\textwidth}@{}}
\hline
Notation & Description & Reference \\
\hline
$M(\mathcal A\otimes\mathbb C)$
&
The complement in $\mathbb C^n$ of the complexification of $\mathcal A$.
&
Section~\ref{sec:Salvetti}
\\

$b\Sigma_{\mathcal A}$
&
The geometric realization of the poset $\operatorname{Fan}(\mathcal A)$;
equivalently, the barycentric subdivision of $\Sigma_{\mathcal A}$.
&
Section~\ref{subsec:zonotope}
\\

$\Sigma_{\mathcal A}$
&
The polyhedral complex dual to the stratification of $\mathbb R^n$
by $\operatorname{Fan}(\mathcal A)$.
&
Section~\ref{subsec:zonotope}
\\

$\widehat{\Sigma}_{\mathcal A}$
&
The Salvetti complex of $\mathcal A$.
&
Section~\ref{sec:Salvetti}
\\

$\widehat F$
&
For a face $F\subseteq\Sigma_{\mathcal A}$, the corresponding
standard subcomplex of $\widehat{\Sigma}_{\mathcal A}$.
&
Section~\ref{sec:Salvetti}
\\

$\Delta_\Lambda$
&
The Artin complex associated with the Artin group $A_\Lambda$.
&
Definition~\ref{def:Artin0}
\\

$C_\Lambda$
&
The Coxeter complex associated with the Coxeter group $W_\Lambda$.
&
Definition~\ref{def:Artin0}
\\

$D_{\mathcal A}$
&
The union of the bounded fans of $\mathcal A$, regarded as a
polyhedral complex.
&
Section~\ref{subsec:falk complexes}
\\

$\widetilde K$
&
The universal cover of the Salvetti complex
$\widehat{\Sigma}_{\mathcal A}$.
&
Definition~\ref{def:falk}
\\

$\mathbb F_{\mathcal A}$
&
The Falk complex of $\mathcal A$; equivalently, the development of
the complex of groups over $D_{\mathcal A}$ described in
Definition~\ref{def:falk complex of groups}.
&
Definitions \ref{def:falk complex of groups} and \ref{def:falk}
\\

$S_{\mathcal A}$
&
For a central arrangement, the cell decomposition of the unit
sphere induced by $\mathcal A$.
&
Section~\ref{subsec:link}
\\

$\Delta_{\mathcal A}$
&
The spherical Deligne complex associated with a central
arrangement $\mathcal A$.
&
Section~\ref{subsec:link}
\\

$S'_{\mathcal A(i)}$,
$\Delta'_{\mathcal A(i)}$
&
The subdivisions of $S_{\mathcal A(i)}$ and
$\Delta_{\mathcal A(i)}$ used for the four types of admissible
local arrangements.
&
Section~\ref{subsec:admissible}
\\

$d_\infty$
&
The metric on $\mathbb F_{\mathcal A}$ obtained from the
$\ell^\infty$ metrics on its faces.
&
Definition~\ref{def:falk metric}
\\
\hline
\end{tabular}
\caption{Notation for the main complexes and metric used in the paper.}
\label{tab:notation}
\end{table}

\section{The Falk complexes and their metric}
\label{subsec:deligne complex}

\subsection{Falk complexes}
\label{subsec:falk complexes}
In \cite{falk1995k}, for each affine arrangement in $\mathbb C^2$ which is the complexification of a real arrangement, Falk described a locally infinite complex which is homotopy equivalent to the associated arrangement complement in $\mathbb C^2$. Falk showed that, compared to other complexes, this complex has the advantage of being easier to arrange a ``non-positive curvature-like'' structure (under some additional assumptions), and hence lead to contractibility results.
Now we describe what we believe should be the correct analogue of Falk's complex in higher dimensions. Although it is no longer true that the complex is homotopy equivalent to the associated complement in $\mathbb C^n$, a weaker statement (see Theorem~\ref{thm:falk} below) holds, which is still useful for proving $K(\pi,1)$ results.

Let $\ca$ be an affine hyperplane arrangement of $\mathbb R^n$. Let $D_\ca$ be the union of all elements in $\fan(\ca)$ which are bounded in $\mathbb R^n$. Note that $D_\ca$ is naturally a stratified space (in the sense of \cite[Chapter II.12.1]{BridsonHaefliger1999}) by considering its fans. We also call fans in $D_\ca$ as \emph{open cells} of $D_\ca$, as $D_\ca$ also has the structure of a polyhedral complex, with each fan being an open cell in this complex.
 A \emph{face} of $D_\ca$ is defined to be the closure of an open cell of $D_\ca$. Each face of $D_\ca$ is a disjoint union of fans. 
 Let $\Sigma_\ca$ be as in Section~\ref{subsec:zonotope}.
 A face $F$ of $D_\ca$ is \emph{dual} to a face $F'$ of $\Si_\ca$ if the barycenter of $F'$ is contained in the interior of $F$. In this case, we will also say $F\subset D_\ca$ is dual to the standard subcomplex $\widehat F'$ of $\od_\ca$. 

Now we define the \emph{Falk complex} of $\mathcal A$. We will give two equivalent definitions. The first one uses the language of complex of groups, and the second one uses only covering space theory (for readers less familiar with complexes of groups).

\begin{definition}
\label{def:falk complex of groups}
We define a simple complex of groups structure $\mathcal U$ on $D_\ca$ as follows (see \cite[Definition II.12.11]{BridsonHaefliger1999} for definitions and terminology). For a face $F$ of $D_\ca$, the local group at $F$ is defined to be the fundamental group of the standard subcomplex of $\od_\ca$ which is dual to $F$.
The morphisms between the local groups are induced by inclusions of the associated subcomplexes. Using the retraction in Definition~\ref{def:retraction}, all the morphisms between the local groups are injective, hence $\mathcal U$ is a simple complex of groups. Moreover, it follows from Definition~\ref{def:retraction} that each local group injects into $\pi_1 (\od_\ca)$. Thus $\mathcal U$ is developable with $\pi_1 (\od_\ca)=\pi_1 \mathcal U$. Let $\falk_\ca$ be the (cellular) development of $\mathcal U$ (see \cite[Theorem II.12.18]{BridsonHaefliger1999}). We call $\falk_\ca$ the \emph{Falk complex of}  $\ca$, motivated by the work of Falk \cite{falk1995k}.
\end{definition}

\begin{definition}[Alternate definition of $\mathbb{F}_\ca$]
\label{def:falk}
    Let $\widetilde K$ be the universal cover of $\od_\ca$. For each face $F$ of $D_{\ca}$, let $F'$ be the face of $\Si_\ca$ dual to $F$. We index the collection of elevations of $\widehat F'$ in $\widetilde K$ using an index set $\Lambda_F$ (recall that an \emph{elevation} of $\widehat F'$ in $\widetilde K$ is a connected component of the inverse image of $\widehat F'$ with respect to the map $\widetilde K\to \od_\ca$). Let $\mathcal F$ be the collection of faces of $D_\ca$.
 Then we define the \emph{Falk complex} of $\ca$ to be
 $$\falk_\ca=\left(\bigsqcup_{F\in \mathcal F} F\times \Lambda_F\right)/\sim,$$
 where we identify $F_1\times \{\lambda_1\}$ as a face of $F_2\times\{\lambda_2\}$ ($F_1,F_2\in\mathcal F$, $\lambda_1\in \Lambda_{F_1}$ and $\lambda_2\in \Lambda_{F_2}$) if $F_1\subset F_2$ in $D_\ca$ and the elevation of $\widehat F'_2$ in $\widetilde K$ associated with $\lambda_2$ is contained in the elevation of $\widehat F'_1$ associated with $\lambda_1$. The action of $\pi_1\od_\ca$ on $\widetilde K$ by deck transformations induces an action $\pi_1\od_\ca\curvearrowright \falk_\ca$, whose quotient is naturally identified with $D_\ca$. This gives a map $q:\falk_\ca\to D_\ca$. 
\end{definition}

 \begin{lem}
 \label{lem:elevation}
\begin{enumerate}
    \item Any elevation of a standard subcomplex $\widehat E$ of $\od_\ca$ is a copy of the universal cover of $\widehat E$.
 \item Let $\widehat E_1$ and $\widehat E_2$ be two standard subcomplexes of $\od_\ca$. For $i=1,2$, let $\widetilde E_i$ be an elevation of $\widehat E_i$ in $\widetilde K$. If $\widetilde E_1\cap\widetilde E_2\neq\emptyset$, then this intersection is an elevation of $\widehat E_1\cap\widehat E_2$.
\end{enumerate}
 \end{lem}

 \begin{proof}
By Definition~\ref{def:retraction}, there is retraction from $\od_\ca$ to $\widehat E$, so $\hat E\to \od_\ca$ is $\pi_1$-injective. Thus (1) follows. The proof of (2) is similar to \cite[Lemma 6.3]{huang2024cycles}, and we provide details here for the convenience of the reader. For (2), it suffices to prove $\widetilde E_1\cap\widetilde E_2$ is connected. Given vertices $x,y\in \widetilde E_1\cap\widetilde E_2$, let $\widetilde P_i$ be a path in $\widetilde E_i$ from $x$ to $y$ for $i=1,2$. Let $P_i$ be the path which is the image of $\widetilde P_i$ in $\widehat E_i$ under the covering map. Then $P_1$ and $P_2$ are homotopic rel endpoints in $\od_\ca$. Let $\Pi_{\widehat E_1}:\od_\ca\to\widehat E_1$ be the retraction map in Definition~\ref{def:falk} and $Q_i=\Pi_{\widehat E_1}(P_i)$. Then $Q_1=P_1$ and $Q_2\subset \widehat E_1\cap \widehat E_2$ by Lemma~\ref{lem:retraction property}. As $Q_1$ and $Q_2$ are homotopic rel endpoints in $\od_\ca$, we know $Q_2$ and $P_2$ are homotopic rel endpoints in $\od_\ca$. Hence $Q_2$ lifts to a path in $\widetilde E_1\cap\widetilde E_2$ connecting $x$ and $y$, as desired.
 \end{proof}

Given an affine hyperplane arrangement $\ca$ of $\mathbb R^n$, and a point $x\in \mathbb R^n$, the \emph{local arrangement} at $x$, denoted by $\ca_x$, is made of the collection of all hyperplanes of $\ca$ that contain $x$. A point $x\in \mathbb R^n$ is an \emph{$\ca$-vertex}, if it is an element of Fan$(\ca)$ (or equivalently, it is a vertex of $D_\ca$).

 \begin{theorem}
\label{thm:falk}
Let $\ca$ be an affine hyperplane arrangement of $\mathbb R^n$. Then its Falk complex $\falk_\ca$ is simply connected.

Suppose in addition that for all $\ca$-vertices of $\mathbb R^n$, the local arrangements at these vertices are $K(\pi,1)$. If the Falk complex $\falk_\ca$ of $\ca$ is contractible, then $\ca$ is $K(\pi,1)$.
\end{theorem}

\begin{proof}
Let $\{x_i\}_{i\in I}$ be the collection of $\ca$-vertices of $\mathbb R^n$. Let $\widehat E_i$ be the standard subcomplex of $\od_\ca$ dual to $x_i$. Then $\widehat E_i$ is the Salvetti complex for the local arrangement $\ca_{x_i}$.
Note that $\{\widehat E_i\}_{i\in I}$ forms a covering of $\od_\ca$. Now we consider the covering $\mathcal U$ of $\widetilde K$ by all possible elevations of elements in $\{\widehat E_i\}_{i\in I}$. Let $\Delta$ be the nerve of $\mathcal U$. By Lemma~\ref{lem:elevation}, each member of $\mathcal U$ is simply connected, and the intersection of finitely many members in $\mathcal U$ is also simply connected (when non-empty). Then by \cite[Theorem 6]{bjorner2003nerves}, $\Delta$ is simply connected.

Elements of $\mathcal U$ are in 1-1 correspondence with vertices of $\falk_\ca$. Moreover, a finite collection of elements in $\mathcal U$ has non-empty common intersection if and only if the associated collection of vertices in $\falk_\ca$ have non-empty common intersection of their open stars. Let $\mathcal U'$ be the covering of $\falk_\ca$ by open stars of vertices of $\falk_\ca$. Then the nerve $\Delta'$ of $\mathcal U'$ is isomorphic to $\Delta$. Moreover, the intersection of finitely many members in $\mathcal U'$ is contractible (if non-empty). Thus by \cite[Theorem 6]{bjorner2003nerves} $\Delta'$ is homotopy equivalent to $\falk_\ca$. Therefore $\falk_\ca$ is simply connected.

If the local arrangements at vertices of $D_\ca$ are $K(\pi,1)$, then $\widehat E_i$ is aspherical. Moreover, any standard subcomplex of $\widehat E_i$ is also aspherical, as it is a retract of $\widehat E_i$ by Definition~\ref{def:retraction}. Thus any finite intersection of elements in $\mathcal U$ is contractible (if non-empty). Hence \cite[Theorem 6]{bjorner2003nerves} implies that $\Delta$ is homotopy equivalent to $\widetilde K$. Thus if $\falk_\ca$ is contractible, then $\widetilde K$ is contractible.
\end{proof}

\subsection{Links of Falk complexes} 
\label{subsec:link}
When $\ca$ is a central arrangement, there is another complex associated with $\ca$, called the \emph{spherical Deligne complex}. This complex was first used implicitly in Deligne's original $K(\pi,1)$ paper \cite{deligne}, then formally defined by Paris under a different name (simply ``the Deligne complex'') \cite{paris1993deligne}. 
The definition that we will use is as follows. 

Let $S_\ca$ be the unit sphere, endowed with the polyhedral complex structure coming from the the intersection of the unit sphere of $\mathbb R^n$ with $\ca$. Then there is a 1-1 correspondence between open cells in $S_\ca$ and elements in $\fan(\ca)$ which are not 0-dimensional. A \emph{face} of $S_\ca$ is the closure of an open cell of $S_\ca$. Each face is a disjoint union of open cells.  
A face $F$ of $S_\ca$ is dual to a face $F'$ of $\Si_\ca$ if the barycenter of $F'$ is contained in the fan associated with $F$. In this case, we also say $F$ is dual to the standard subcomplex $\widehat F'$ of $\od_\ca$. 
 The definition of the spherical Deligne complex of $\ca$, denoted by $\sd_\ca$, is identical to Definition~\ref{def:falk}, except we consider faces of $S_\ca$ instead of faces of $D_\ca$. Similarly, we have a natural action $\pi_1\od_\ca\curvearrowright \sd_\ca$ which induces a map $\sd_\ca\to S_\ca$.

\begin{lem}
\label{lem:iso}
Suppose $\ca$ is the reflection arrangement associated with a finite Coxeter group with Coxeter diagram $\Lambda$. Then $\Delta_\ca$ is isomorphic to the Artin complex $\Delta_\Lambda$.
\end{lem}

\begin{proof}
By Theorem~\ref{thm:reflection arrangement}, the 1-skeleton of the universal cover of $\od_\ca$ can be identified with the Cayley graph of $A_\Lambda$. As $\Si_\ca$ is dual to the Coxeter complex $\bC_\Lambda$, and the 1-skeleton of $\od_\ca$ is obtained from the 1-skeleton of $\Si_\ca$ by replacing each edge by a pair of edges, we deduce that
the 1-skeleton $\od_\ca$ is naturally identified with the Cayley graph of $W_\Lambda$. Hence each edge of $\od_\ca$ is labeled by a generator of $W_\Lambda$, i.e. a vertex of $\Lambda$. Let $S$ be the collection of vertices of $\Lambda$. Given a face $F'$ of $\Si_\ca$ with associated standard subcomplex $\widehat F'$ of $\od_\ca$, let $T$ be the set of labels of edges in $\widehat F'$. By Definition~\ref{def:retraction}, $\widehat F'\to \od_\ca$ is $\pi_1$-injective. Thus elevations of $\widehat F'$ in the universal cover of $\od_\ca$ are in 1-1 correspondence with left cosets of $A_T$ (i.e. the standard parabolic subgroup of $A_\Lambda$ generated by $T$) in $A_\Lambda$. 

On the other hand, we define the \emph{type} of a simplex in the Artin complex $\Delta_\Lambda$ to be the intersection of the types of its vertices. Thus there is a 1-1 correspondence between simplices in $\Delta_\Lambda$ of type $T$ and left cosets of $A_T$ in $\Delta_\Lambda$. Now the lemma follows from the definition of $\Delta_\Lambda$ and $\Delta_\ca$.
\end{proof}

Let $\ca$ be an arbitrary affine hyperplane arrangement. Let $x$ be a vertex in $D_\ca$, with $\ca_x$ be the local arrangement at $x$. Then $\lk(x,D_\ca)$ can be naturally identified with a subcomplex of $S_{\ca_x}$ as in Lemma~\ref{lem:link deligne}, which is a consequence of the description of $\falk_\ca$ and $\sd_{\ca_x}$.
 \begin{lem}(\cite[Lemma 4.10]{huang2024cycles})
 	\label{lem:link deligne}
 Let $x'\in \falk_\ca$ be a vertex which maps to $x\in D_\ca$ under $\falk_\ca\to D_\ca$. Let $N$ be the inverse image of $\lk(x,D_\ca)$ (viewed as a subset of $S_{\ca_x}$) under the map $\sd_{\ca_x}\to S_{\ca_x}$. Then $\lk(x',\falk_\ca)\cong N$.
 \end{lem}
 
\subsection{Metrizing Falk complexes}
\begin{definition}
\label{def:falk metric}
Suppose $\ca$ is an admissible affine arrangement in $\mathbb R^n$, and let $\falk_\ca$ be the associated Falk complex. We define a metric $d_\infty$ on $\falk_\ca$ as follows.
Let $D_\ca$ be defined as in the previous section. As $D_\ca$ is a subset of $\mathbb R^n$, we endow $D_\ca$ with the $\ell^\infty$ metric $d_{\infty}$ on $\mathbb R^n$. This restricts to a metric on each face of $D_\ca$, hence an $\ell^\infty$-metric on each face of $\falk_\ca$ using the map $\falk_\ca\to D_\ca$. Then we can define a pseudo-metric $d_\infty$ on $\falk_\ca$ by considering the infimum of lengths of strings between each pair of points as in Section~\ref{subsec:injective}.
\end{definition}

\begin{lem}
\label{lem:geodesic metric}
The pseudometric $d_\infty$ is a metric. Moreover, $(\falk_\ca,d_\infty)$ is a complete geodesic metric space.
\end{lem}

\begin{proof}
Note that if $\falk_\ca$ has only finitely many isometry types of its closed cells, then the lemma follows from the same argument in \cite[p.~101, Theorem I.7.13 and p.~105, Theorem I.7.19]{BridsonHaefliger1999}. In the more general case, note that the map $(\falk_\ca,d_\infty)\to (D_\ca,d_\infty)$ is 1-Lipschitz. So for any $x\in \falk_\ca$, the ball $B(x,r)$ in $(\falk_\ca,d_\infty)$ is mapped to a bounded region in $(D_\ca,d_\infty)$. So the smallest subcomplex of $\falk_\ca$ containing $B(x,r)$ has only finitely many isometry types of its closed cells. As the lemma only concerns properties that need to be verified on each ball, this finishes the proof.
\end{proof}

\subsection{Falk complexes for admissible arrangements}
\label{subsec:admissible}
For $i=1,2$, let $\ca_i$ be an affine hyperplane arrangement in $\mathbb R^{m_i}$. Then the \emph{product} of $\ca_1$ and $\ca_2$ is an arrangement $\ca$ in $\mathbb R^{m_1+m_2}$ whose hyperplanes are of form $H\times \mathbb R^{m_2}$ for $H\in \ca_1$ or $\mathbb R^{m_1}\times H$ for $H\in \ca_2$.  
\begin{definition}
\label{def:admissible}
We say an affine hyperplane arrangement $\ca$ in $\mathbb R^n$ is \emph{admissible}, if at each $\mathcal A$-vertex $x$, the local arrangement at $x$ is a translate of one of the following four types:
\begin{enumerate}
    \item (type $B_n$) $x_i\pm x_j=0$ for $1\le i\neq j\le n$ and $x_i=0$ for $1\le i\le n$;
    \item (type $D_n$) $x_i\pm x_j=0$ for $1\le i\neq j\le n$;
    \item (skewed type $A_n$) $x_i=0$ for $1\le i\le n$ and $x_i=x_j$ for $1\le i\neq j\le n$, or any image of this this arrangement under the $(\bb{Z}/2\bb{Z})^n$ action on $\bb[n]{R}$ by reflections about the coordinate hyperplanes; or
    \item a product of the previous types.
\end{enumerate}
\end{definition}

Note that a skewed type $A_1$ arrangement is just the hyperplane $x=0$ in $\mathbb R$.

For $1\le i\le 4$, let $\ca(i)$ be the central arrangement in item $(i)$ of the above list. Let $S_{\ca(i)}$ and $\Delta_{\ca(i)}$ be defined in Section~\ref{subsec:link}. Note that for $i>1$, we can subdivide $S_{\ca(i)}$ into $S'_{\ca(i)}:=S_{\ca(1)}$ since $\ca(i) \subseteq \ca(1)$. Let $\Delta'_{\ca(i)}$ be the subdivision of $\Delta_{\ca(i)}$ obtained from pulling back the subdivision $S'_{\ca(i)}$ of $S_{\ca(i)}$ via the map $\Delta_{\ca(i)}\to S_{\ca(i)}$. Let $\Delta'_{\ca(1)}:=\Delta_{\ca(1)}$.

\begin{definition} \label{def:u type s type}
Let $1\le i\le 3$. The \emph{$u$-type} (short for ``unsubdivided type'') of a vertex in $S_{\ca(i)}$ is the usual type of the vertex viewed as a Coxeter complex as labeled in Section~\ref{subsec:coxeter and artin}. For $1 \leq i \leq 4$, the \emph{$s$-type} (``subdivided type'') of a vertex in $S'_{\ca(i)}$ is the type of the corresponding vertex in $S_{\ca(1)}$ viewed as the Coxeter complex $\bC_\Lambda$  of type $\Lambda = B_n$, also as labeled in Section~\ref{subsec:coxeter and artin}.
The $u$-type of a vertex of $\Delta_{\ca(i)}$ is its usual type from the spherical Deligne complex, and the $s$-type of a vertex in $\Delta'_{\ca(i)}$ is defined to be the pull-back of $s$-type of vertices in $S'_{\ca(i)}$. %

We define a relation $<$ on $(\Delta'_{\ca(i)})^0$. Given two distinct vertices $v,w\in (\Delta'_{\ca(i)})^0$ of ($s$-)type $\hat s_j,\hat s_k$ respectively, define $v<w$ if $v$ and $w$ are adjacent in $\Delta'_{\ca(i)}$ and $j<k$. As is common convention, we say $v \leq w$ if $v < w$ or $v = w$. We call this the \emph{$s$-order} and sometimes write $\leq_s$. We call the usual ordering on $\sd_{\ca(i)}^0$ the \emph{$u$-order} and sometimes write $\leq_u$. 
\end{definition}

\begin{lem} \label{lem:1 2 3 poset}
    For $1\le i\le 3$ the $s$-order $\leq_s$ on $\Delta'_{\ca(i)}$ is a partial order.
\end{lem}
\begin{proof}
    The case $i=1$ is verified in \cite[Lem 6.5]{haettel2021lattices}, the case $i=2$ is verified in \cite[Lem 2.22]{huang2024cycles}, and the case $i = 3$ is verified in Proposition~\ref{prop:type 3 partial order}.
\end{proof}

For the type (4) arrangements, we have a nice description of the relation $\leq$.

\begin{lem} \label{lem:decomposing 4}
    Suppose $\cb = \cb_1 \times \dots \times \cb_k$, where each $\cb_j$ is an arrangement of type (1), (2), or (3). For each $i$, let $(\c P_i,<)$ denote $((\Delta'_{\cb_i})^0,<)$ with a smallest element $0$ added.
    Then $((\Delta'_{\cb})^0,<)$ is isomorphic to $(\c P_1 \times \dots \times \c P_k \setminus (0,\dots,0), <)$ with the product relation defined by $(p_1\dots, p_n) < (p'_1,\dots,p'_n)$ if $p_i < p'_i$ for each $i$. 
    
    In particular, the $s$-order on $\Delta'_{\ca(4)}$ is a partial order.
\end{lem}

\begin{proof}
    For each $j$, we know $\c P_j$ is a poset by Lemma~\ref{lem:1 2 3 poset}. The geometric realization $|\mathcal P_i|$ of $\mathcal P_i$ is a cone over $\Delta'_{\cb_i}$. Let $\mathcal P=\mathcal P_i\times\dots\times \mathcal P_k$. 
    Then $\mathcal P$ is a poset (under the product order) whose geometric realization $|\mathcal P|$ is a product $|\mathcal P_1|\times \dots \times |\mathcal P_k|$ of spaces. On the other hand, as $\cb$ is a product of the $\cb_i$, we know the cone over $\Delta'_{\cb}$ is a product of the cones over the $\Delta'_{\cb_i}$. Thus $((\Delta'_{\cb})^0,<)$ is isomorphic to $\mathcal P$ with the smallest element $(0,\dots,0)$ removed, which is still a poset. 
\end{proof}

We can summarize these lemmas as

\begin{prop} \label{prop:s is partial order}
    For $1\le i\le 4$ the $s$-order $\leq$ on $\Delta'_{\ca(i)}$ is a partial order.
\end{prop}

\begin{lem}
\label{lem:basic properties}
For $1\le i\le 4$, a collection of vertices of $\Delta'_{\ca(i)}$ span a simplex of $\Delta'_{\ca(i)}$ if and only if they correspond to a chain in $((\Delta'_{\ca(i)})^0,<)$.
\end{lem}

\begin{proof}
The forward implication follows from Proposition~\ref{prop:s is partial order}; namely, if $V$ is a set of vertices which span a simplex, then they are pairwise adjacent, hence pairwise comparable, and thus form a chain since $\leq$ is a partial order. The reverse implication will follow from the fact that each subdivided complex is flag. Assuming this, if $V$ is a chain, then the elements are pairwise comparable, hence pairwise adjacent, and hence span a simplex. 
We now see why these are flag complexes. For $i = 1$, since the subdivision is trivial, this follows from the general fact for Artin complexes \cite[Prop.~4.5]{godelle2012k}. 
For $i = 2$, this follows easily from the description of the subdivided complex for $D_n$ type given in Section \ref{subsec:coxeter and artin}.
For $i = 3$, this is Proposition~\ref{prop:A3 simplex chain}.
The case $i = 4$ follows from (the proof of) Lemma~\ref{lem:decomposing 4}: 
the cone over $\Delta'_{\cb}$ is a product of the cones over the $\Delta'_{\cb_i}$, the cone over a flag complex is flag, the product of flag complexes is flag, and the link of a vertex in a flag complex is flag. Since the link of the cone point is isomorphic to $\Delta'_\cb$, the result follows.
\end{proof}

An affine hyperplane arrangement $\ca$ in $\mathbb R^n$ is \emph{complete} if $D_\ca=\mathbb R^n$. We say $\ca$ has \emph{finite shape}, if there are only finitely many isometry types of cells in $D_\ca$. Note that if $\ca$ is invariant under the action of a translation group $\mathbb Z^n$, then $\ca$ is both complete and has finite shape.

\begin{lem}
\label{lem:link to injective}
Suppose $\ca$ is a complete admissible affine arrangement in $\mathbb R^n$ with finite shape. Suppose for $1\le i\le 4$, the poset $((\Delta'_{\ca(i)})^0,<)$ is bowtie free and upward flag. Then $(\falk_\ca,d_\infty)$ is an injective metric space. In particular, $\falk_\ca$ is contractible.
\end{lem}

\begin{proof}
Let $\pi:\falk_\ca\to D_\ca$ be the 1-Lipschitz projection map. Given $x\in \falk_\ca$ mapping to an $\ca$-vertex $\bar x\in D_\ca$, let $\epsilon(x)$ be a positive number such that the ball $B(\bar x,\epsilon(x))$ in $D_\ca$ is contained in the open star of $\bar x$ in $D_\ca$ (here the ball is taken with respect to the $\ell^\infty$-metric). As $\pi$ is 1-Lipschitz, $B(x,\epsilon(x))$ is the connected component of $\pi^{-1}(B(\bar x,\epsilon(x)))$ that contains $x$. Suppose the local arrangement at $\bar x$ is $\ca(i)$ (up to a translation). Then the subdivision of $S_{\ca(i)}$ described as above induces a subdivision of $B(\bar x,\epsilon(x))$ (which is a cube with side length $2\epsilon(x)$) into orthoschemes of size $\epsilon(x)$. Hence $B(x,\epsilon(x))$ is also subdivided into a simplicial complex made up of orthoschemes of size $\epsilon(x)$. Let $\mathcal P$ be the poset obtained by adding a smallest element to $((\Delta'_{\ca(i)})^0,<)$. Then by Lemma~\ref{lem:link deligne} and Lemma~\ref{lem:basic properties}, the simplicial structure on $B(x,\epsilon(x))$ is isomorphic to the geometric realization $|\mathcal P|$ of $\mathcal P$. Such an isomorphism induces a natural bijection $f$ from $B(x,\epsilon(x))$ to the $\epsilon(x)$-ball around the cone point in $|\mathcal P|_\infty$. Note that by our choice of $\epsilon(x)$, in order to measure $d_\infty(y_1,y_2)$ for $y_1,y_2\in B(x,\epsilon(x))$, we only need to consider strings from $y_1,y_2$ that are contained in the open star of $x$ in $\falk_\ca$, as strings which escape the open star will automatically have length $\ge 2\epsilon (x)\ge d_\infty(y_1,y_2)$. Thus %
$f$ is an isometry. By our assumptions and Theorem~\ref{thm:injective criterion}, $|\mathcal P|_\infty$ is an injective metric space, hence so are any of its balls, and in particular $B(x,\epsilon(x))$ is an injective metric space.

We claim for each open face $F\subset \falk_\ca$, there is an $\epsilon>0$, depending only on the dimension of $F$, such that for each $x\in F$, $B(x,\epsilon)$ is injective. We induct on the dimension of $F$. The case where $F$ is a vertex is treated in the previous paragraph. Now suppose $\dim(F)=n > 0$. As $\partial F$ is a disjoint union of finitely many open faces of lower dimension, by induction there is an $\epsilon>0$ such that $B(x,\epsilon)$ is injective for each $x\in \partial F$. Thus for any $x\in F$ such that $d(x,\partial F)\le \epsilon/4$, $B(x,\epsilon/4)$ is injective as $B(x,\epsilon/4)\subset B(y,\epsilon)$ for some $y\in \partial F$, and any ball in an injective metric space is injective. 
If $x\in F$ is such that $d(x,\partial F)\ge \epsilon/4$, by a similar consideration of the position of strings as in the previous paragraph, there is an $\epsilon'>0$ such that the $B(x,\epsilon')$ are isometric to each other for all such $x$. Thus we can assume $d(x,\partial F)=\epsilon/4$ and finish as before. Given two different open faces $F_1$ and $F_2$ of $D_\ca$, we write $F_1\sim F_2$ if $F_1$ and $F_2$ are isometric, and each hyperplane of $\ca$ containing $F_1$ is parallel to a hyperplane of $\ca$ containing $F_2$ and vice versa. The relation $\sim$ is transitive and divides the collection of $n$-dimensional open faces of $D_\ca$ into finitely many equivalent classes by the finite shape assumption. Moreover, if $F_1\sim F_2$, $x_1\in F_1$, and $x_2\in F_2$, then a small ball around $x_1$ and a small ball around $x_2$ are isometric. Thus we can use the same $\epsilon$ for $n$-dimensional open faces in the same equivalence class. This proves the claim.

It follows that $\falk_\ca$ is locally uniformly injective. As $\falk_\ca$ is simply connected by Theorem~\ref{thm:falk}, we deduce from Theorem~\ref{thm:local-to-global} that $\falk_\ca$ is injective, hence contractible.
\end{proof}

We postpone the proof of the following key proposition to Section \ref{sec:check links}.
\begin{prop}
\label{prop:key}
The poset $((\Delta'_{\ca(3)})^0,<)$ is bowtie free and upward flag.
\end{prop}

\begin{remark}
\label{rmk:contrast}
    There is an interesting contrast between this proposition and Theorem~\ref{thm:bowtie free An}, as the poset in Theorem~\ref{thm:bowtie free An} is not upward flag.
\end{remark}

\begin{thm}
\label{thm:main}
Suppose $\ca$ is a complete admissible affine arrangement in $\mathbb R^n$ with finite shape. Suppose Conjecture~\ref{conj:dn} holds. Then $(\falk_\ca,d_\infty)$ is an injective metric space and $\ca$ is a $K(\pi,1)$ arrangement.
\end{thm}

\begin{proof}
Proposition~\ref{prop:key}, Conjecture~\ref{conj:dn} and Theorem~\ref{thm:Bnflag} together imply that $((\Delta'_{\ca(i)})^0,<)$ is bowtie free and upward flag for $i=1$, $2$, and $3$, respectively. The bowtie free and upward flag properties are not changed by adding or deleting a smallest element, or by taking products. So by Lemma~\ref{lem:decomposing 4}, $((\Delta'_{\ca(i)})^0,<)$ is bowtie free and upward flag when $i = 4$.
Then the theorem follows from Lemma~\ref{lem:link to injective} and Theorem~\ref{thm:falk}. 
\end{proof}

\begin{cor}
Suppose $\ca$ is a complete finite shape admissible affine arrangement in $\mathbb R^n$ with $n=2,3,4$. Then $\ca$ is a $K(\pi,1)$ arrangement.
\end{cor}

\begin{proof}
This follows from Theorem~\ref{thm:main} and Theorem~\ref{thm:dn dim 3 and 4}.
\end{proof}

The next result can be proved in the same way as Theorem~\ref{thm:main}, using Proposition~\ref{prop:key} and Theorem~\ref{thm:Bnflag}.
\begin{cor}
\label{cor:AB}
Suppose $\ca$ is a complete finite shape affine arrangement in $\mathbb R^n$ such that for each $\mathcal A$-vertex $x$, the local arrangement at $x$ is a translate of one of the following three types:
\begin{enumerate}
    \item (type $B_n$) $x_i\pm x_j=0$ for $1\le i\neq j\le n$ and $x_i=0$ for $1\le i\le n$;
    \item (skewed type $A_n$) $x_i=0$ for $1\le i\le n$ and $x_i=x_j$ for $1\le i\neq j\le n$, or any image of this this arrangement under the $(\bb{Z}/2\bb{Z})^n$ action on $\bb[n]{R}$ by reflections about the coordinate hyperplanes; or
    \item a product of the previous types.
\end{enumerate}
Then $(\falk_\ca,d_\infty)$ is an injective metric space and $\ca$ is a $K(\pi,1)$ arrangement.
\end{cor}

\section{Some examples of admissible arrangements}
\label{sec:example}
In each dimension, we consider an infinite family of finite affine arrangements, and an infinite family of infinite complete affine arrangements, as follows.

\begin{definition}
For $k\ge 1$, let $\ch_{k,n}$ be the affine hyperplane arrangement in $\mathbb R^n$ given by $x_i \in \{-2k-1, -2k+1, \dots,-3, -1, 1, 3, \dots, 2k - 1, 2k + 1\}$ for $1 \le i \le n$ and $x_i \pm x_j = 0$ for $1\le i\neq j\le n$.

For $k\ge 1$, let $\ck_{k,n}$ be the affine hyperplane arrangement in $\mathbb R^n$ given by $x_i\in \mathbb Z$ for $1\le i\le n$, and $x_i+x_j \in 2k\mathbb Z+1$, $x_i-x_j \in 2k\mathbb Z$ for $1\le i\neq j\le n$.
\end{definition}

\begin{prop}
\label{prop:admissible}
Both $\ch_{k,n}$ and $\ck_{k,n}$ are admissible in the sense of Definition~\ref{def:admissible}.
\end{prop}

\begin{proof}
We only treat the family $\ck_{k,n}$, as the family $\ch_{k,n}$ is similar and much simpler. An $(x_i,x_j)^+$ hyperplane of $\ck_{k,n}$ is a hyperplane defined by $x_i+x_j\in 2k\mathbb Z+1$, and an $(x_i,x_j)^-$ hyperplane of $\ck_{k,n}$ is a hyperplane defined by $x_i-x_j\in 2k\mathbb Z$. An $(x_i,x_j)$ hyperplane is either an $(x_i,x_j)^+$ hyperplane or an $(x_i,x_j)^-$ hyperplane.
Let $\theta$ be a $\ck_{k,n}$-vertex.

\medskip
\noindent
\underline{Case 1: all the coordinates of $\theta$ are integers.} Then the local arrangement at $\theta$ contains all the coordinate hyperplanes. We define $x_i\sim x_j$ if there is an $(x_i,x_j)$ hyperplane containing $\theta$. We claim $\sim$ is an equivalence relation. Suppose $x_i-x_j=2kn_1$ and $x_j-x_k=2kn_2$ contain $\theta$. Then $x_i-x_k=2k(n_1+n_2)$ contains $\theta$. Suppose $x_i-x_j=2kn_1$ and $x_j+x_k=2kn_2+1$ contain $\theta$. Then $x_i+x_k=2k(n_1+n_2)+1$ contains $\theta$. Suppose $x_i+x_j=2kn_1+1$ and $x_j+x_k=2kn_2+1$ contain $\theta$. Then $x_i-x_k=2k(n_2-n_1)$ contains $\theta$. Thus the claim is proved. As $\theta$ has integer coordinates, for each $i\neq j$, $\theta$ is not contained simultaneously in an $(x_i,x_j)^-$ hyperplane and an $(x_i,x_j)^+$ hyperplane. Thus the local arrangement at $\theta$ is a product of skewed type $A$ arrangements, one for each equivalent class.

\medskip
\noindent
\underline{Case 2: all the coordinates of $\theta$ are not integers}. Then the coordinates are of form $1/2+k\mathbb Z$. We say the $x_i$-coordinate of $\theta$ is \emph{even}, if it is $1/2+k\cdot even$. Otherwise the $x_i$-coordinate of $\theta$ is \emph{odd}. For $i\neq j$, if the $x_i$-coordinate and $x_j$-coordinate of $\theta$ have the same parity (i.e. they are both odd or both even), then we have both an $(x_i,x_j)^-$ hyperplane and an $(x_i,x_j)^+$ hyperplane containing $\theta$. If the $x_i$-coordinate and $x_j$-coordinate of $\theta$ have different parity, then there are no $(x_i,x_j)^-$ hyperplanes and no $(x_i,x_j)^+$ hyperplanes containing $\theta$. Thus the local arrangement at $\theta$ is a product of two arrangements of type $D$, one involving the even coordinates, and another one involving the odd coordinates.

\medskip
\noindent
\underline{Case 3: the coordinates of $\theta$ has both integers and non-integers.} Note that if $x_i$ is an integer coordinate, and $x_j$ is a non-integer coordinate, then there are no $(x_i,x_j)^-$ hyperplanes and no $(x_i,x_j)^+$ hyperplanes containing $\theta$. Thus 
the local arrangement at $\theta$ is a product of the arrangements in Case 1 (between the integer coordinates) and the arrangements in Case 2 (between the non-integer coordinates).
\end{proof}

\begin{remark}
Both $\ch_{k,n}$ and $\ck_{k,n}$ are not of fiber type in an obvious way. In fact, the arrangements $\ch_{k,n}$ are not supersolvable for any $k\ge 1$ and $n\ge 2$. 

For an intuitive perspective of this fact, note that for each parallel family of hyperplanes in $\ck_{k,n}$, one can find $\ck_{k,n}$-vertices that are not contained in any member of this parallel family. A similar discussion applies to $\ch_{k,n}$. So one cannot produce an iterated fibration structure on these arrangement complements in a similar way as when dealing with supersolvable arrangement complements. 

Now we give a more rigorous proof that $\ch_{k,n}$ are not supersolvable. Note that restricting $\ch_{k,n}$ to the affine subspace $x_3=\cdots=x_n=1$ gives precisely $\ch_{k,2}$. As supersolvability is preserved under restriction \cite[Corollary 2.7]{AmendHogeRoehrle2014}, it suffices to show that $\ch_{k,2}$ is not supersolvable. The projective completion of $\ch_{k,2}$ has no modular intersection point. Indeed, a finite modular point would have to lie on both diagonal lines and on a vertical line, which is impossible. Moreover, none of the four intersection points at infinity is modular. Hence $\ch_{k,2}$, and consequently $\ch_{k,n}$, is not supersolvable.

Alternatively, the localization of the arrangement $\ch_{k,n}$ at the origin is an arrangement of type $D_n$, which is not supersolvable when $n\ge 4$; see  \cite[Theorem 1.2]{HogeRoehrle2014} and \cite[Example 5.5]{JambuTerao1984}. It follows from \cite[Lemma 2]{Stanley1971} that localizations of supersolvable arrangements are supersolvable.

The family $\ch_{k,n}$ generalizes \cite[Example 3.13]{falk1995k} which is neither supersolvable nor simplicial. 
\end{remark}

To show $\ch_{k,n}$ and $\ck_{k,n}$ are $K(\pi,1)$ arrangements, we need the following which is a consequence of \cite[Lem 5.5]{huang2024cycles} and Theorem~\ref{thm:Bnflag}.

\begin{prop}
\label{prop:positive part}
Let $\ca(3)$, $S'_{\ca(3)}$ and $\Delta'_{\ca(3)}$ be as in Section~\ref{subsec:admissible}. Let $(S'_{\ca(3)})^+$ be the part of the sphere $S'_{\ca(3)}$ in the first orthant (i.e. all coordinates are $\ge 0$). Let $(\Delta'_{\ca(3)})^+$ be the inverse image of $(S'_{\ca(3)})^+$ under $\Delta'_{\ca(3)}\to S'_{\ca(3)}$. Then the restriction of the partial order  $(\Delta'_{\ca(3)})^0$ to the vertex set of $(\Delta'_{\ca(3)})^+$ is bowtie free and upward flag. 
\end{prop}

\begin{thm}
\label{thm:examples}
Suppose Conjecture~\ref{conj:dn} holds in dimension $n$. Then $(\falk_{\ch_{k,n}},d_\infty)$ and $(\falk_{\ck_{k,n}},d_\infty)$ are injective metric spaces, and 
$\ch_{k,n}$ and $\ck_{k,n}$ are $K(\pi,1)$ arrangements for any $k\ge 1$. 

Thus by Theorem~\ref{thm:dn dim 3 and 4}, for $n=2,3,4$ and any $k\ge 1$, the arrangements $\ch_{k,n}$ and $\ck_{k,n}$ are $K(\pi,1)$ arrangements. 
\end{thm}

\begin{proof}
The $\ck_{k,n}$ case is a consequence of Theorem~\ref{thm:main} and Proposition~\ref{prop:admissible}, as it is complete. The arrangement $\ch_{k,n}$ is not complete, but we will explain how it is still a $K(\pi,1)$ arrangement using techniques similar to the proof for complete arrangements. Let $D$ be the union of all elements in $\fan(\ch_{k,n})$ that are bounded. Then $D$ is the cube bounded by $\{x_i=\pm (2k+1)\}_{1\le i\le n}$. Let $\falk$ be the associated Falk complex. We can prove $(\falk,d_\infty)$ is injective in the same way as Lemma~\ref{lem:link to injective}, except we need to verify that for each $\ca$-vertex $\bar x$ on the boundary of $D$ and $x\in \falk$ that is mapped to $\bar x$, $B(x,r)$ with the induced metric from $(\falk,d_\infty)$ is injective. Note that the local arrangement $\ca_{\bar x}$ at $\bar x$ is a product of one skewed type $A_m$ arrangement $\mathcal L_1$ and a few skewed type $A_1$ arrangements $\{\mathcal L_i\}_{i=2}^m$. Thus $S_{\ca_{\bar x}}$ admits a join decomposition 
$$
S_{\ca_{\bar x}}=S_{\mathcal L_1}\circ S_{\mathcal L_2}\circ\cdots\circ S_{\mathcal L_m}.
$$
Note that $S_{\mathcal L_i}$ is a copy of $\mathbb S^0$ for $i\ge 2$. 
Let $\bar N=\lk(\bar x,D)$. Then $\bar N=\bar N_1\circ\cdots\circ \bar N_k$, where $\bar N_1$ is the subset of $S_{\mathcal L_1}$ in the first orthant (up to reflections along the coordinate hyperplanes). For $i\ge 2$, either $\bar N_i=S_{\mathcal L_i}$ or $\bar N_i$ is a single point. Now by Theorem~\ref{thm:injective criterion}, Lemma~\ref{lem:link deligne}, Proposition~\ref{prop:positive part} and the argument in the proof Theorem~\ref{thm:main} for multiple factors, we know $B(x,r)$ is injective for $r$ small enough.
\end{proof}

\section{The skewed \texorpdfstring{$A_n$}{A\_n} arrangement}
\label{sec:check links}

Our goal in this section is to prove Proposition~\ref{prop:contractible} (= Proposition~\ref{prop:key}). We will begin with a more in-depth discussion of the different subdivisions $S_{\ca(3)}$ and $S'_{\ca(3)}$ of the sphere (introduced in Section~\ref{subsec:admissible}) in the specific case of type (3), how these subdivisions of the sphere relate, and how their properties extend to the complexes $\Delta_{\ca(3)}$ and $\Delta_{\ca(3)}'$. 
For ease of notation, \textbf{in the rest of this section, let} $\cb = \ca(3)$. Thus $S_\cb,S'_\cb,\Delta_\cb,\Delta'_\cb$ denote the corresponding four objects in Table~\ref{tab:notation}.

\subsection{The Coxeter complex}

It becomes convenient to replace $S_\cb$ and $S'_\cb$ with their projections to the boundary of the $n$-cube $[-1,1]^n$, inheriting the cell structures of  $S_\cb$ and $S'_\cb$, resp. (not the cell structure of the cube).
Equivalently, we can redefine $S_\cb$ and $S'_\cb$ replacing the unit sphere by $\partial [-1,1]^n$.
For $S'_\cb$, this gives the barycentric subdivision of the usual cell structure on $\partial [-1,1]^n$. See Figures \ref{fig:unsubdivided} and \ref{fig:subdivided}.

\begin{figure}[h!]
    \centering
    \includegraphics[width=0.3\linewidth]{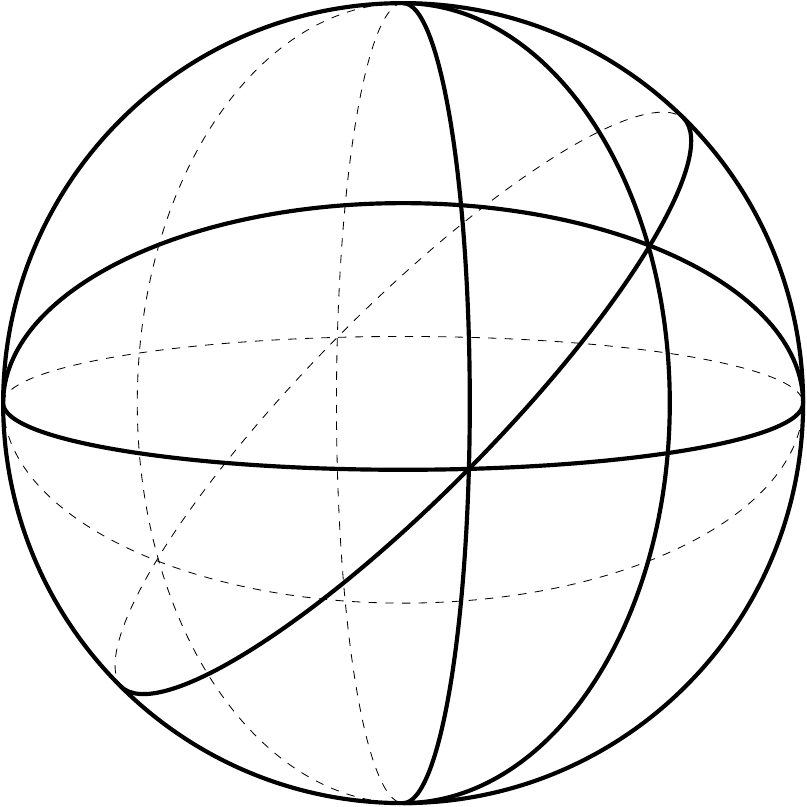} 
    \hspace{0.1\linewidth}
    \includegraphics[width=0.3\linewidth]{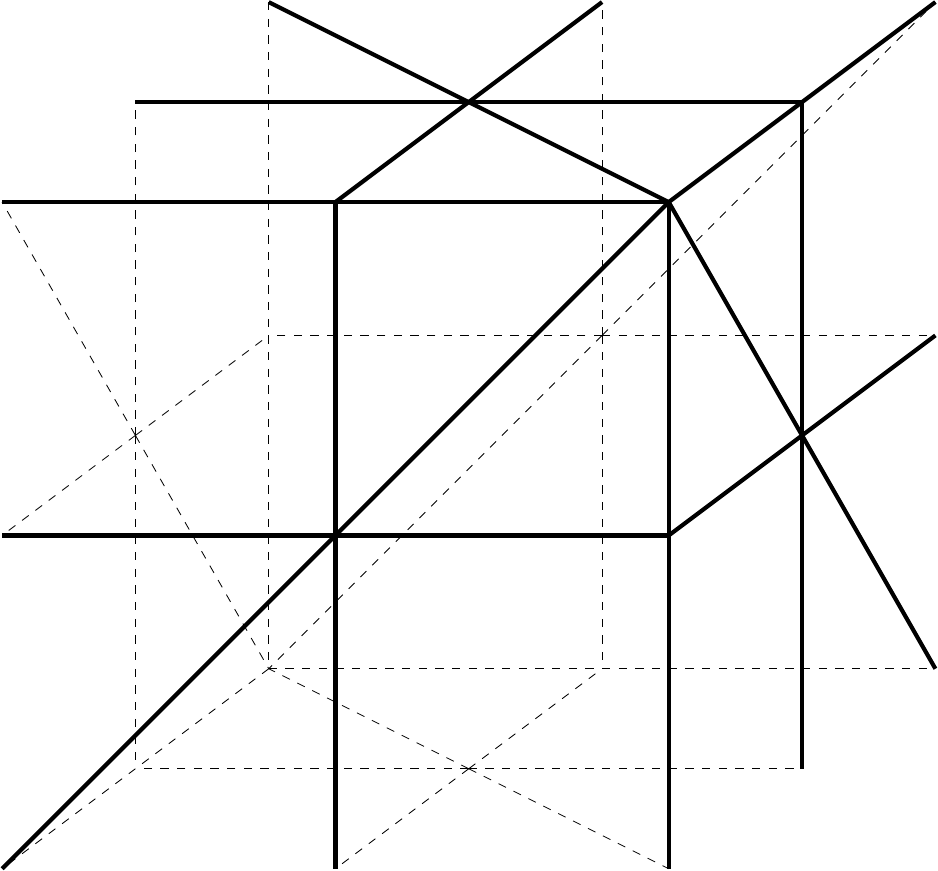}
    \caption{$S_\cb$ on the $2$-sphere and its projection to the boundary of the $3$-cube}
    \label{fig:unsubdivided}
\end{figure}

\begin{figure}[h!]
    \centering
    \includegraphics[width=0.3\linewidth]{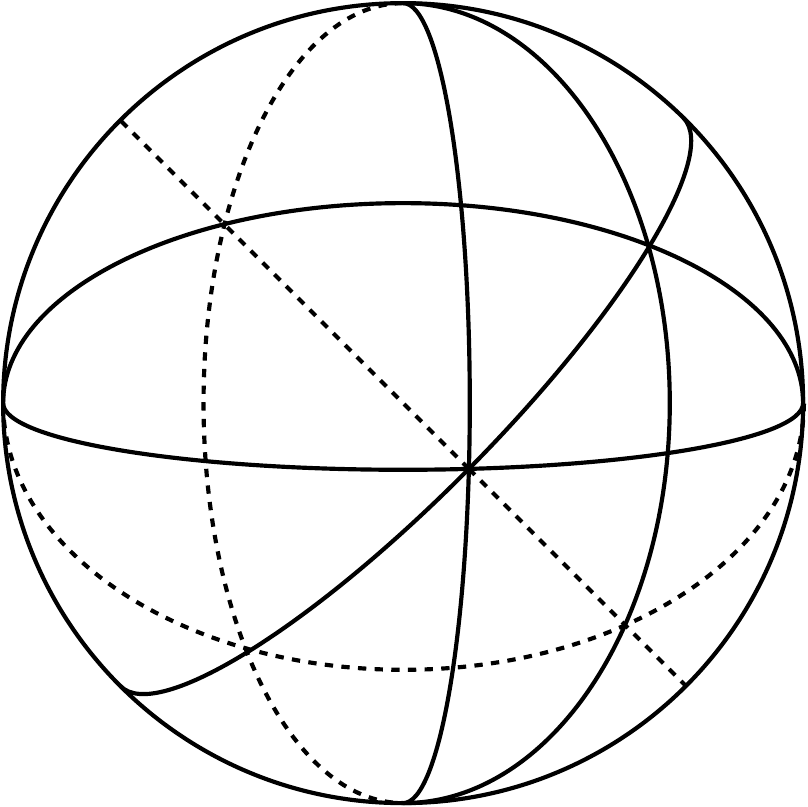} 
    \hspace{0.1\linewidth}
    \includegraphics[width=0.3\linewidth]{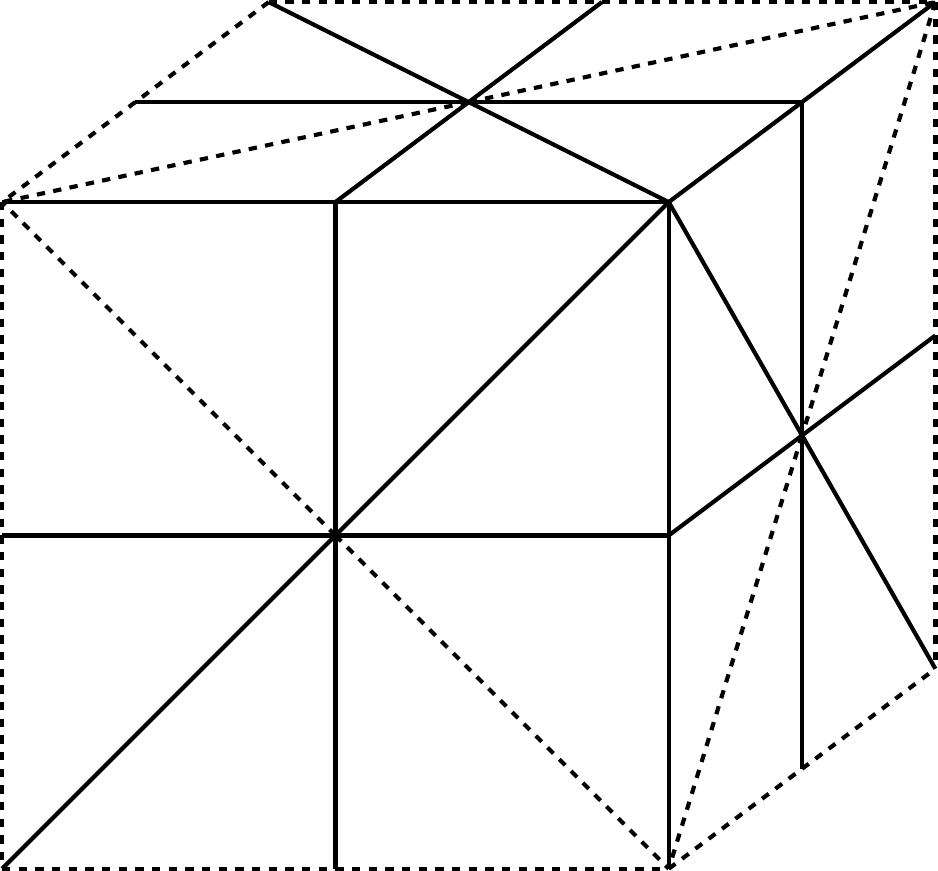}
    \caption{The subdivision $S_\cb'$ on $2$-sphere and its projection to the boundary of the $3$-cube. New edges are represented by dashed lines.}
    \label{fig:subdivided}
\end{figure}

We briefly recall the labeling of the vertices of $S_\cb$ and $S_\cb'$. The general definition is given in Definition~\ref{def:u type s type}.
The $u$-type of a vertex $v$ of $S_\cb$ is of the form $\hat u_i$; we define $u(v) = i$. Similarly, the $s$-type of a vertex $w$ of $S_\cb'$ has the form $\hat s_j$; we define $s(w) = j$.

There are two orderings we consider: the $u$-order and the $s$-order. The $u$-order is defined on vertices of $S_\cb$, where we say $v <_u w$ if $v \sim_u w$ (meaning $v$ and $w$ are adjacent in the \emph{u}nsubdivided complex $S_\cb$) and $u(v) < u(w)$. The $s$-order is defined on vertices of $S_\cb'$, where we say $v <_s w$ if $v \sim_s w$ (meaning $v$ and $w$ are adjacent in the \emph{s}ubdivided complex $S_\cb'$) and $s(v) < s(w)$. As is common, we say $v \leq_u w$ if either $v = w$ or $v <_u w$, and $v \leq_s w$ if either $v = w$ or $v <_s w$. Sometimes we may write $v \simeq_u w$ to say $v = w$ or $v \sim_u w$, and may write $v \simeq_s w$ to say $v = w$ or $v \sim_s w$. 

Notice that if $v$ is a vertex of $S_\cb$ then it is also a vertex of $S_\cb'$, so it has a $u$-type and an $s$-type, and these types do not necessarily agree. In particular there may be vertices $v$ and $w$ with $v \leq_u w$ but $w \leq_s v$ (or even $v \not\sim_s w$, so the vertices are not comparable in the $s$-order). We clarify the connection in the following.

For a vertex $p = (p_1,\dots,p_n)$ in $S'_\cb$, let $\pos(p) = \{\, i : p_i > 0 \,\} = \{\, i : p_i = 1 \,\}$ and $\neg(p) = \{\, i : p_i < 0 \,\} = \{\, i : p_i = -1 \,\}$ (recall we are projecting to the cube). The following facts follow quickly from the definitions and are left as an exercise.

\begin{prop} 
\label{prop:basic coxeter}
Let $v$ be a vertex of $S'_\cb$.
    \begin{itemize}
        \item $s(v) = \#\pos(v) + \#\neg(v)$.
        \item If $w$ is another vertex of $S'_\cb$, then $v$ and $w$ are adjacent if and only if either $\pos(v) \subset \pos (w)$ and $\neg(v)\subset \neg(w)$, or $\pos(w)\subset \pos(v)$ and $\neg(w)\subset \neg(v)$.
        \item $v$ is a vertex of $S_\cb$ if and only if either $\neg(v) = \varnothing$ or $\pos(v) = \varnothing$. \textup{(In the first case, we will call $v$ ``non-negative'', and in the second, we will call $v$ ``non-positive''.)}
        \item If $v$ is non-negative, then $u(v) = s(v) = \#\pos(v)$.
        \item If $v$ is non-positive, then $u(v) + s(v) = n+1$, or in other terms, $u(v) = n+1 - s(v) = n+1 - \#\neg(v)$.
    \end{itemize}
\end{prop}

See Figure \ref{fig:labels on sb} for an example of the labeling of $S_\cb$ and $S_\cb'$ in \bb[3]{R}.

\begin{figure}[h!]
    \centering
    \includegraphics[width=0.35\linewidth]{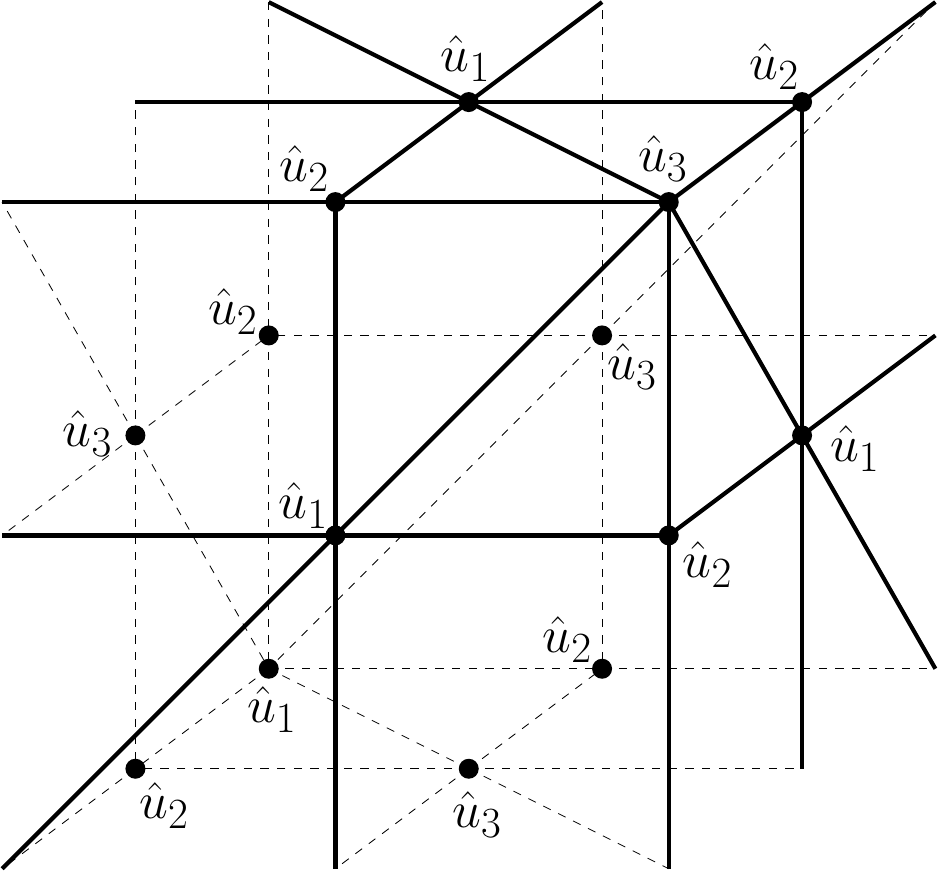}
    \hspace{0.1\linewidth}
    \includegraphics[width=0.36\linewidth]{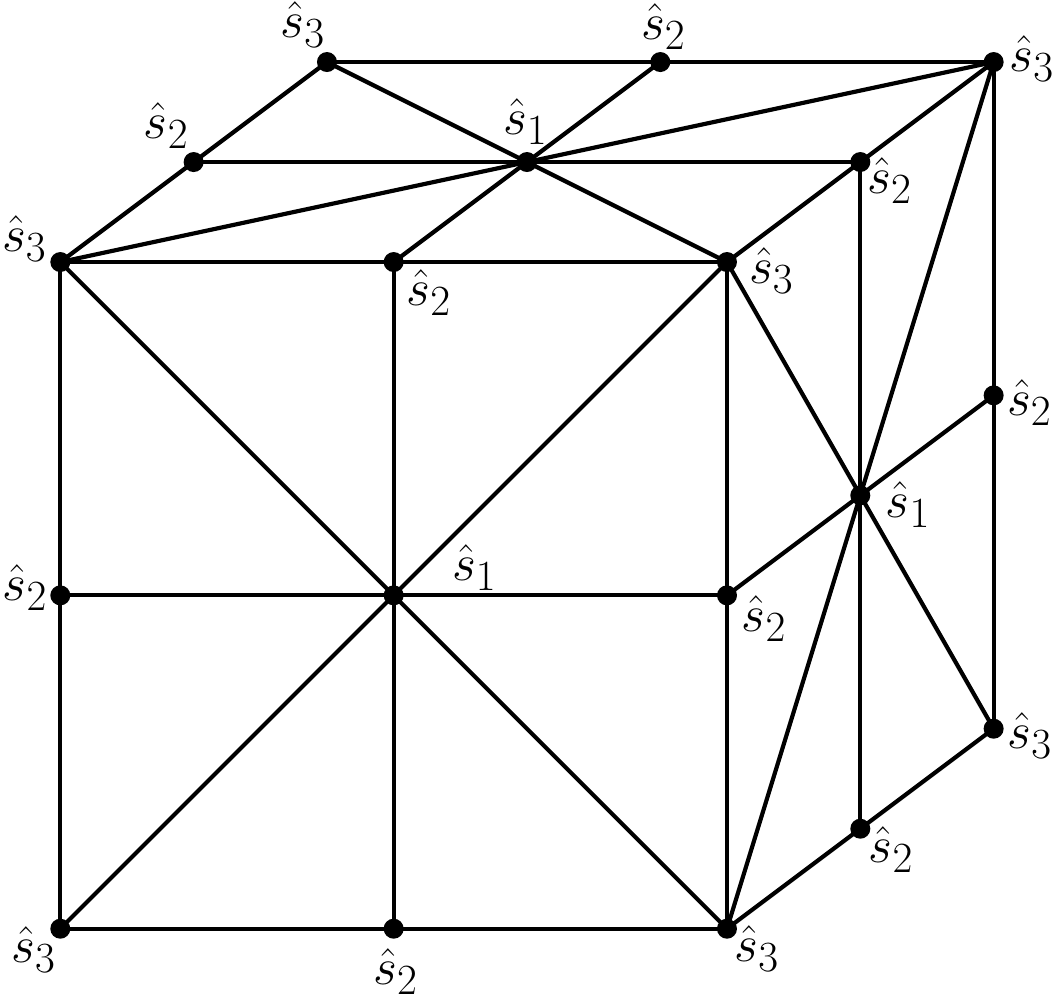}
    \caption{The $u$-labeling on $S_\cb$ and the $s$-labeling on $S_\cb'$}
    \label{fig:labels on sb}
\end{figure}

Note that there are always vertices (and generally, cells) in $S'_\cb$ which are not vertices (resp., cells) of $S_\cb$.
For a subspace $\delta$ of the $n$-sphere, if $\delta$ is a closed cell of $S_\cb$, we say $\delta$ is a \emph{real simplex}, and if $\delta$ is a closed cell of $S_\cb'$ which is not a cell of $S_\cb$, we say $\delta$ is a \emph{fake simplex}. Note that by definition, a simplex cannot be both real and fake; if $\delta$ is both a cell of $S_\cb$ and the subdivision $S_\cb'$, it is still only considered a real simplex. 

Adjacency in $S_\cb$ is slightly more subtle than adjacency in $S'_\cb$.

\begin{lem} \label{lem:adjacency in coxeter}
Let $v$ and $w$ be vertices of $S_\cb$. If both $v$ and $w$ are non-negative or both non-positive, then they are adjacent in $S_\cb$ if and only if they are adjacent in $S'_\cb$. In this case, there is no fake vertex on the edge between them. If $v$ is non-negative  and $w$ is non-positive, then they are adjacent in $S_\cb$ if and only if their non-zero coordinates do not overlap. In this case, $u(v) < u(w)$, and there is a fake vertex $b$ laying on the real edge between $v$ and $w$ satisfying $\pos(b) = \pos(v)$ and $\neg(b) = \neg(w)$.
\end{lem}

\begin{proof}
    If both $v$ and $w$ are non-negative or both non-positive, then they are both in the same orthant, and there is no subdivision within this orthant, so adjacency in the unsubdivided complex is equivalent to adjacency in the subdivided complex.

    Now let $v = (v_1,\dots,v_n)$ be a non-negative vertex and $w = (w_1,\dots,w_n)$ a non-positive vertex. Notice that as a basic consequence of the definition of $S_\cb$, $v$ and $w$ are adjacent in $S_\cb$ if and only if every hyperplane $H$ containing neither $v$ nor $w$ has $v$ and $w$ contained in the same connected component of its complement $\bb[n]{R} \setminus H$. This statement is equivalent to saying that $v$ and $w$ are vertices of the same chamber of the arrangement, which in turn is equivalent to adjacency because the arrangement is simplicial, and every edge is contained in a chamber.

If $v$ and $w$ have non-overlapping non-zero coordinates, then $v_i>0$ implies $w_i=0$ and $w_i<0$ implies $v_i=0$, hence $v$ and $w$ are contained in the same closed halfspace bounded by $\{x_i=0\}$. Moreover, if $v_i>v_j$ for $i\neq j$ then $v_i>0$ and $v_j=0$. Consequently, $w_i=0$, $w_j\le 0$, and $w_i\ge w_j$. Similarly, $w_i>w_j$ implies $v_i\ge v_j$. So $v$ and $w$ are contained in the same closed halfspace bounded by $\{x_i=x_j\}$ for $i\neq j$. Hence $v$ and $w$ are contained in the same chamber, hence adjacent.

    Conversely, suppose $v$ and $w$ are adjacent. 
    Then for every hyperplane $H$, we either have that $H$ contains $v$, $H$ contains $w$, both, or $v$ and $w$ are on the same side of the complement of $H$. Suppose $H = \{x_i = 0\}$ is a coordinate hyperplane for some $i$. Then $v_i \geq 0$ and $w_i \leq 0$, so they cannot lie in the same component of the complement of $H$ (since these are the halfspaces $x_i > 0$ and $x_i < 0$). This means $H$ must contain at least one of $v_i$ or $w_i$, or in other words, we must have at least one of $v_i = 0$ or $w_i = 0$. Since this holds for each $i$, it follows that the non-zero coordinates do not overlap, as claimed.

    To see that $u(v) < u(w)$, since the non-zero coordinates do not overlap, i.e., $\pos(v) \cap \pos(w) = \varnothing$, we know that $\#\pos(v) + \#\neg(w) \leq n$. Then $u(v) = \#\pos(v) \leq n - \#\neg(w) < n + 1 - \#\neg(w) = u(w)$.

    Define $b = (b_1,\dots,b_n)$ by 
    \[
        b_i = \begin{cases}
            b_i = v_i & v_i \not= 0 \\
            b_i = w_i & w_i \not= 0 \\
            b_i = 0 & \text{otherwise} \\
        \end{cases}.
    \]
    Since $\pos(v) \cap \neg(w) = \varnothing$, this is well-defined.
    It is a straightforward linear algebra exercise to see that $b$  lies on the edge of $S_\cb$ between $v$ and $w$. 
\end{proof}

This gives information on adjacent real vertices, but we can also say something about a real vertex which is adjacent to a fake vertex.

\begin{lem} \label{lem:fake adjacent real s type coxeter}
    If $b$ is a fake vertex adjacent to a real vertex $a$ in $S'_\cb$, then $s(b) > s(a)$.
\end{lem}

\begin{proof}
    Since $b$ is a fake vertex, it must have both positive and negative coordinates, but since $a$ is real, it can only have one type. If $\pos(a) = \varnothing$, then we cannot have $\pos(b) \subseteq \pos(a)$, so by the adjacency rules of $S'_\cb$, we must have $\pos(a) \subsetneq \pos(b)$ and $\neg(a) \subseteq \neg(b)$ (where the first inclusion is proper). Similarly, if $\neg(a) = \varnothing$, then we cannot have $\neg(b) \subseteq \neg(a)$, so we have $\pos(a) \subseteq \pos(b)$ and $\neg(a) \subsetneq \neg(b)$ (where the second inclusion is proper).
    In either case, 
    \[
        s(a) = \#\pos(a) + \#\neg(a) < \#\pos(b) + \#\neg(b) = s(b). \qedhere
    \]
\end{proof}

For any vertex $b = (b_1,\dots,b_n)$ of $S'_\cb$, let $b^+ = (b_1^+,\dots, b_n^+)$ be given by $b_i^+ = b_i$ when $b_i \geq 0$ and $b_i^+ = 0$ when $b_i < 0$, and similarly, let $b^- = (b_1^-,\dots, b_n^-)$ be given by $b_i^- = b_i$ when $b_i \leq 0$ and $b_i^- = 0$ when $b_i > 0$. 
The next Lemma follows immediately from Lemma~\ref{lem:adjacency in coxeter}. 

\begin{lem} \label{lem:coxeter fake projections}
    If $b$ is a fake vertex, then $b^+$ and $b^-$ are real vertices which are adjacent in $S_\cb$, and $b$ lies on the real edge joining $b^+$ and $b^-$. 
\end{lem}

The following two propositions follow from Theorem~\ref{thm:bowtie free An} and Theorem~\ref{thm:Bnflag}. %
\begin{prop} \label{prop:coxeter poset s}
    For vertices $a,b \in S'_\cb$, say $a \leq_s b$ if $a$ and $b$ are adjacent or equal, and $s(a) \leq s(b)$. Then $((S'_\cb)^0, \leq_s)$ is a poset.
\end{prop}
\begin{prop}\label{prop:coxeter poset u}
    For vertices $a,b \in S_\cb$, say $a \leq_u b$ if $a$ and $b$ are adjacent or equal, and $u(a) \leq u(b)$. Then $((S_\cb)^0, \leq_u)$ is a poset.
\end{prop}

\subsection{The Artin complex}

Our main goal in this and the upcoming sections is to transfer what we know about $S_\cb$, $S_\cb'$, and $\sd_\cb$ to the subdivided complex $\sd_\cb'$ in order to discuss injectivity of its metric. In the current section, we want to show that the relation we have placed on $(\sd_\cb')^0$ is in fact a partial order. Our techniques broadly follow two main ideas. First, we can use the projection from the Artin complex to the Coxeter complex to leverage basic facts about adjacency derived in the previous section. Second, we can use a ``projection'' from fake vertices to real vertices to leverage the fact that we already know that the relation on $(\sd_\cb)^0$ is a partial order. We start by recalling previous relevant definitions and notation.

Recall that we view $\sd_\cb$ and $\sd_\cb'$ as different cell structures on the same underlying space, as described in the paragraph before Definition \ref{def:u type s type}.
Recall also from Definition \ref{def:u type s type} that vertices have an $s$-type, $u$-type, or both. If $v$ is a vertex of $\sd_\cb$, it has a $u$-type $\hat u_i$ and we define $u(v) = i$. If $w$ is a vertex of $\sd_\cb'$, it has an $s$-type $\hat s_j$ and we define $s(w) = j$.

Similar to the definitions in $S_\cb$ and $S_\cb'$, if $\delta$ is a subspace of $\sd_\cb$ or $\sd_\cb'$, then, if $\delta$ is a closed cell of $\sd_\cb$, we say $\delta$ is a \emph{real simplex}, and if $\delta$ a closed cell of $\sd'_\cb$ which is not a cell of $\sd_\cb$, we say $\delta$ is a \emph{fake simplex}. %
Note that, in particular, every real vertex is also a vertex of $\sd_\cb'$ (but not every real simplex is also a simplex of $\sd_\cb'$).

If $v$ and $w$ are vertices which are adjacent (resp., adjacent or equal) in the unsubdivided complex $\sd_\cb$, we will write $v \sim_u w$ (resp., $v \simeq_u w$). 
If $v$ and $w$ are adjacent (resp., adjacent or equal) in the subdivided complex $\sd_\cb'$, we will write $v \sim_s w$ (resp., $v \simeq_s w$). Recall that we say $v <_u w$ (resp.~$v\leq_uw$) if $v \sim_u w$ (resp.~$v\simeq_uw$) and $u(v) \leq u(w)$, and similarly we say $v <_s w$ (resp.~$v\leq_sw$) if $v \sim_s w$ (resp.~$v\simeq_sw$) and $s(v) \leq s(w)$. .

Let $\overline{\,\cdot\,}$ be the projection from the Artin complex to the Coxeter complex, as defined in Section \ref{subsec:link}. 
For a vertex $v \in \sd'_\cb$, we will let $\pos(v) = \pos(\overline v)$ and $\neg(v) = \neg(\overline v)$.

The following are either straightforward exercises or immediate consequences of the definitions (note that there is no extra branching along fake simplices). Figure~\ref{fig:unsubdivided}, Figure~\ref{fig:subdivided} and Figure~\ref{fig:labels on sb} might be helpful for understanding the Proposition in low dimensions.
In particular, we note that all branching occurs in real simplices, and added fake simplices do not introduce any added branching.

\begin{prop} \label{prop:basic artin}
Let $v$ be a vertex of $\sd_\cb'$.
    \begin{enumerate}
        \item \label{item:projection injective} The projection $\overline{\,\cdot\,}$ is injective on both real and fake simplices.
        \item A cell $\delta$ is real (resp., fake) if and only if $\overline \delta$ is real (resp., fake).
        \item $s(v) = s(\overline v)$, and if $v$ is real, $u(v) = u(\overline v)$ (in other words, the projection is type-preserving on both $\sd_\cb$ and $\sd_\cb'$).
        \item $v$ is a real vertex if and only if either $\neg(v) = \varnothing$ or $\pos(v) = \varnothing$. \textup{(In the first case, we will call $v$ ``non-negative'', and in the second, we will call $v$ ``non-positive''.)}
        \item If $v$ is non-negative, then $u(v) = s(v) = \# \pos(v)$, and if $v$ is non-positive, $u(v) = n+1 - s(v)$.
    \end{enumerate}
\end{prop}

We now need to establish some facts about adjacency in $\sd_\cb$ and $\sd_\cb'$ similar to those found in the Coxeter complexes. The main technique is to reduce to the setting where each vertex is contained in a common simplex, then project and apply Proposition \ref{prop:basic artin}(\ref{item:projection injective}).

\begin{lem} \label{lem:adjacency in deligne}
    Suppose $v$ and $w$ are vertices of $\sd_\cb'$.
    \begin{enumerate}
        \item \label{lem:adjindel1} If $v$ and $w$ are contained in a common real simplex, then $v \simeq_s w$ if and only if either $\pos(v) \subseteq \pos(w)$ and $\neg(v) \subseteq \neg(w)$, or $\pos(w) \subseteq \pos(v)$ and $\neg(w) \subseteq \neg(v)$.
        \item If both are non-negative (or both non-positive), then $v$ and $w$ are adjacent in $\sd_\cb$ if and only if they are adjacent in $\sd_\cb'$.
        \item If $v$ is non-negative, $w$ is non-positive, and $v \sim_u w$, then their non-zero coordinates do not overlap. In this case, $v <_u w$, and there is a fake vertex $y$ on the edge between $v$ and $w$. 
    \end{enumerate}
\end{lem}

\begin{proof}
    Points (1) and (3) follow immediately from Proposition~\ref{prop:basic artin}(\ref{item:projection injective}), Proposition~\ref{prop:basic coxeter}, and Lemma~\ref{lem:adjacency in coxeter}. 
    For (2), 
    first suppose $v$ and $w$ are adjacent in $\sd_\cb$. Then the real edge $e$ between them injectively maps to the real edge $\overline e$ of $S_\cb$ between the non-negative (or non-positive) vertices $\overline v$ and $\overline w$ (Proposition~\ref{prop:basic artin}(\ref{item:projection injective}), (2), and (4)). Then Lemma~\ref{lem:adjacency in coxeter} says there is no fake vertex on $\overline e$, so there is no fake vertex on $e$, and hence $v \sim_s w$.

    Conversely, suppose  $v$ and $w$ are adjacent in $\sd_\cb'$. 
    Let $e$ be the edge of $\sd'_\cb$ connecting $v$ and $w$. Then there is a real simplex $\delta \subseteq \sd_\cb$ containing $e$: if $e$ is a real edge, choose $\delta = e$, and if $e$ is a fake edge, then since $\sd'_\cb$ is a (strictly) finer cell structure than $\sd_\cb$, it must be properly contained in some real simplex.
    Then $\delta$ maps isomorphically to $\overline \delta$ (Proposition~\ref{prop:basic artin}(\ref{item:projection injective})). Then $\overline v$ and $\overline w$ are both non-negative (or both non-positive) and are adjacent in $S'_\cb$, so they are adjacent in $S_\cb$ by Lemma~\ref{lem:adjacency in coxeter}. In particular, $\overline e$ must have been a real edge to start with, since there are no double edges, and thus so is $e$, implying $v \sim_u w$.
\end{proof}

The next Lemma gives a slight strengthening of one direction of Lemma \ref{lem:adjacency in deligne}(\ref{lem:adjindel1}) with reference to the ordering, rather than just adjacency. 

\begin{lem} \label{lem:leq s containment}
    For vertices $v$ and $w$ of $\sd_\cb'$, if $v \leq_s w$, then $\pos(v) \subseteq \pos(w)$ and $\neg(v) \subseteq \neg(w)$.
\end{lem}
\begin{proof}
    Suppose to the contrary that either $\pos(v) \not\subseteq \pos(w)$ or $\neg(v) \not\subseteq \neg(w)$. If it is also the case that $\pos(w) \not\subseteq \pos(v)$ or $\neg(w) \not\subseteq \neg(v)$, then $v \not\sim_s w$ (Lemma~\ref{lem:adjacency in deligne}), which is a contradiction. So $\pos(w) \subseteq \pos(v)$ and $\neg(w) \subseteq \neg(v)$. Our original contradiction assumption tells us that (at least) one of these inclusions must be proper, so we may assume that $\pos(w) \subsetneq \pos(v)$ (the other inclusion results in an identical conclusion). Then $\#\pos(w) < \# \pos(v)$ and $\#\neg(w) \leq \#\neg(v)$, so 
    \[
        s(w) = \#\pos(w) + \#\neg(w) < \#\pos(v) + \#\neg(v) = s(v),
    \]
    which contradicts the fact that $s(v) \leq s(w)$ (from the definition of $\leq_s$).
\end{proof}

\begin{lem} \label{lem:fake not less than real}
    Suppose $v$ is a real vertex and $w$ is a fake vertex. If $v \sim_s w$, then $v <_s w$.
\end{lem}

\begin{proof}
    The edge between $v$ and $w$ is contained in a real simplex $\delta$. The projection is injective on $\delta$, so $\overline v$ and $\overline w$ are adjacent in $S'_\cb$. Moreover, $\overline v$ is a real vertex and $\overline w$ is a fake vertex. So by Lemma~\ref{lem:fake adjacent real s type coxeter} and Proposition~\ref{prop:basic artin}(3), $s(v) = s(\overline v) < s(\overline w) = s(w)$, and thus $v <_s w$.
\end{proof}

Similar to the Coxeter complex, we can define ``projections'' from fake vertices to distinguished adjacent real vertices.

\begin{lem} \label{lem:fake projects to real}
    Suppose $v$ is a fake vertex. Then there exists a (necessarily unique) non-negative vertex $v^+$ and a (unique) non-positive vertex $v^-$ such that 
    \begin{enumerate}
        \item $v^+ \sim_u v^-$;
        \item $v$ lies on the real edge between $v^+$ and $v^-$; and
        \item $(\overline v)^+ = \overline{v^+}$ and $(\overline v)^- = \overline{v^-}$.
    \end{enumerate}
    Moreover, $s(v) = s(v^+) + s(v^-) = u(v^+) - u(v^-) + n+1$.
\end{lem}

\begin{proof}
    Similar to the previous Lemmas, this is a straightforward application of Proposition \ref{prop:basic artin}, this time with Lemma \ref{lem:coxeter fake projections}. In more detail,
    let $\delta$ be a real simplex containing $v$. Then $\delta$ is combinatorially isomorphic to $\overline \delta$ and $\overline v$ is a fake vertex of $S'_\cb$. We know that $\overline v$ lies on the real edge $\overline e$ between $(\overline v)^+$ and $(\overline v)^-$. Since $\overline\delta$ is a real simplex and contains an interior point of the real edge $\overline e$, it must contain all of $\overline e$, and in particular, contains $(\overline v)^+$ and $(\overline v)^-$. Letting $e$, $v^+$, and $v^-$ be the preimages of $\overline e$, $(\overline v)^+$, and $(\overline v)^-$, resp., under the isomorphism $\delta \to \overline \delta$ and applying Proposition~\ref{prop:basic artin}(3) gives the result. 
\end{proof}

We now characterize $s$-adjacency and order of fake vertices in terms of the $u$-adjacency and order of their real projections. 

\begin{lem} \label{lem:fake projections are adjacent}
    Suppose $v$ is a real vertex, $w$ is a fake vertex, and $v \sim_s w$. If $v$ is non-negative, then $v \leq_u w^+$ and $v \leq_u w^-$. If $v$ is non-positive, then $w^+ \leq_u v$ and $w^- \leq_u v$.
\end{lem}

\begin{proof}
We first show that $v \simeq_u w^+$ and $v \simeq_u w^-$%
. Let $e$ be the real edge between $w^+$ and $w^-$, so that $w$ lies in the interior of $e$. Let $e'$ be the fake edge between $v$ and $w$. This fake edge must be contained in a real simplex $\delta$. Then this simplex contains the midpoint of the real edge $e$, so it must contain the whole edge, and in particular, must contain its vertices $w^+$ and $w^-$. Since this real simplex also contains $v$, it follows that $v \simeq_u w^+$ and $v \simeq_u w^-$. 

Assume that $v$ is non-negative. 
Since $w^-$ is non-positive and $v \sim_u w^-$, we know $v \leq_u w^-$ (Lemma~\ref{lem:adjacency in deligne}(3)).
Since $v \leq_s w$ (Lemma~\ref{lem:fake not less than real}), then $\pos(v) \subseteq \pos(w) = \pos(w^+)$ (Lemma~\ref{lem:leq s containment}). Since $v$ and $w^+$ are non-negative, $u(v) = \#\pos(v) \leq \#\pos(w^+) = u(w^+)$, so $v \leq_u w^+$.

Now assume that $v$ is non-positive. 
Since $w^+$ is non-negative and $v \sim_u w^+$, we know $w^+ \leq_u v$ (Lemma~\ref{lem:adjacency in deligne}(3)).
Since $v \leq_s w$ (Lemma~\ref{lem:fake not less than real}), then $\neg(v) \subseteq \neg(w) = \neg(w^-)$ (Lemma~\ref{lem:leq s containment}), and in particular, $\#\neg(v) \leq \#\neg(w^-)$. Since $v$ and $w^-$ are non-positive, $u(v) = n+1-\#\neg(v) \geq n+1-\#\neg(w^-) = u(w^-)$, so $w^- \leq_u v$.
\end{proof}

\begin{lem} \label{lem:chain of fake}
    Suppose $v$ and $w$ are fake vertices. Then $v \leq_s w$ if and only if $v^+ \leq_u w^+ \leq_u w^-  \leq_u v^-$.
\end{lem}

\begin{proof}
    Suppose $v \leq_s w$. Let $e_v$ be the real edge between $v^+$ and $v^-$, let $e_w$ be the real edge between $w^+$ and $w^-$, and let $e_{vw}$ be the fake edge between $v$ and $w$. Then $e_{vw}$ must be contained in a real simplex $\delta$. Since $\delta$ contains interior points of $e_v$ and $e_w$, it must contain all of $e_v$ and all of $e_w$, and in particular, contains $v^+$, $v^-$, $w^+$, and $w^-$. This means that $v^+$, $v^-$, $w^+$, and $w^-$ are pairwise $u$-adjacent. 
    Since $w^+$ is non-negative and $v^-$ is non-positive,  $w^+ \leq_u v^-$. 
    By Lemma~\ref{lem:leq s containment}, since $v \leq_s w$,
    \begin{align*}
        \pos(v^+) &= \pos(v) \subseteq \pos(w) = \pos(w^+), \quad \text{ and} \\
        \neg(v^-) &= \neg(v) \subseteq \neg(w) = \neg(w^-).
    \end{align*}
    Then 
    \begin{align*}
        u(v^+) &= \#\pos(v^+) \leq \#\pos(w^+) = u(w^+), \quad \text{ and} \\
        u(v^-) &= n + 1 - \#\neg(v) \geq n + 1 - \#\neg(w^+) = u(w^-).
    \end{align*}

    Conversely, suppose $v^+ \leq_u w^+ \leq_u w^-  \leq_u v^-$. 
    Since the $u$-order is transitive, the vertices $v^+$, $v^-$, $w^+$, and $w^-$ are pairwise $u$-adjacent. Since $\sd_\cb$ is a flag complex \cite[Prop.~4.5]{godelle2012k}, these vertices span a real simplex. This real simplex contains the real edges which contain $v$ and $w$, so the simplex contains $v$ and $w$. Now, since $v^+ \leq_u w^+$ and $v^+$ and $w^+$ are non-negative, we have $v^+ \leq_s w^+$ (Lemma~\ref{lem:adjacency in deligne}(1) and (2), and Proposition~\ref{prop:basic artin}(5)). Then $\pos(v^+) \subseteq \pos(w^+)$ (Lemma~\ref{lem:leq s containment}). 
    Similarly, $w^- \leq_u v^-$, so $w^- \sim_s v^-$ since they are both non-positive (Lemma~\ref{lem:adjacency in deligne}(2)), and $u(w^-) \leq u(v^-)$, so $s(v^-) = n + 1 - u(v^-) \leq n + 1 - u(w^-) = s(w^-)$, implying $v^- \leq_s w^-$, so $\neg(v^-) \subseteq \neg(w^-)$. 
    Then 
    \begin{align*}
        \pos(v) &= \pos(v^+) \subseteq \pos(w^+) = \pos(w) \qquad \text{ and } \\ \neg(v) &= \neg(v^-) \subseteq \neg(w^-) = \neg(w),
    \end{align*}
    so since $v$ and $w$ are contained in a real simplex, $v \leq_s w$ by Lemma~\ref{lem:adjacency in deligne}(1).
\end{proof}

The next Proposition is the main application of this section, and the key Proposition for showing $\leq_s$ is a partial order.

\begin{prop} \label{prop:comparable contained in real}
    Let $V = \{v_1,\dots,v_k\}$ be a collection of vertices of $\sd_\cb'$ such that $v_i \leq_s v_{i+1}$ for each $i$. Then there exists a real simplex $\delta$ containing $V$.
\end{prop}

\begin{proof}
    By Lemma~\ref{lem:fake not less than real}, there must exist some $0 \leq j \leq k$ so that $v_i$ is real when $i \leq j$ and is fake when $i > j$.
    Let $V_r = \{v_1,\dots,v_{j}\}$ denote the (possibly empty) set of real vertices contained in $V$, and let $V_f = \{v_{j+1}\dots,v_k\}$ denote the (possibly empty) set of fake vertices.
    Let $V_{f\pm} = \{\,v_i^+, v_i^- : i > j \,\}$ be the real projections of the fake vertices. 
    By Lemma~\ref{lem:chain of fake}, for each $j < i < k$, $v_i^+ \leq_u v_{i+1}^+ \leq_u v_{i+1}^-  \leq_u v_i^-$. By transitivity of the $u$-order, this means 
    \[
    v_{j+1}^+ \leq_u v_{j+2}^+ \leq_u \ldots \leq_u v_k^+ \leq v_k^- \leq_u v_{k-1}^- \leq_u \ldots  \leq_u v_{j+1}^-.
    \]
    
    The adjacency within the set of real vertices $V_r$ occurs in the subdivided complex $\sd_{\cb}'$, so every vertex of $V_r$ is either non-negative or non-positive.
    Suppose first that they are non-negative. Then $u(v_i) = s(v_i) < s(v_{i+1}) = u(v_{i+1})$ for every $i < j$, so $v_i \leq_u v_{i+1}$. 
    Moreover, since $v_{j}$ is non-negative and adjacent to the fake vertex $v_{j+1}$, Lemma~\ref{lem:fake projections are adjacent} tells us that $v_j \leq_u v_{j+1}^+$. Since the $u$-order is a partial order, and in particular transitive, this means
    \[
    v_1 \leq_u \ldots \leq_u v_j \leq v_{j+1}^+  \leq_u \ldots \leq_u v_k^+ \leq v_k^- \leq_u v_{k-1}^- \leq_u \ldots  \leq_u v_{j+1}^-.
    \]
    In other words, the elements of 
    $V_r \cup V_{f\pm}$ are real vertices which are pairwise connected by real edges. Since the unsubdivided complex is flag, this implies that they span a real simplex. The real edges between the elements of $V_{f\pm}$ contain the fake vertices $V_f$, so this real simplex contains all of $V$.
    The non-positive case is similar. The only difference is that 
    $u(v_i) = n + 1 - s(v_i) > n + 1 - s(v_{i+1}) = u(v_{i+1})$ if $i < j$, so $v_{i+1} \leq_u v_i$, but since $v_j$ is non-positive, Lemma~\ref{lem:fake projections are adjacent} now tells us that $v_{j+1}^- \leq_u v_j$, so we still have a linearly ordered set 
    \[
    v_{j+1}^+  \leq_u \ldots \leq_u v_k^+ \leq v_k^- \leq_u v_{k-1}^- \leq_u \ldots  \leq_u v_{j+1}^- \leq_u v_j \leq_u v_{j-1} \leq_u \ldots \leq_u v_1.
    \]
    Then for the same reasons in the non-negative case, $V_r \cup V_{f\pm}$ spans a real simplex containing $V$.
\end{proof}

\begin{prop} \label{prop:type 3 partial order}
    The relation $\leq$ is a partial order on $(\sd_{\ca(3)}')^0 = (\sd_\cb')^0$. 
\end{prop}

\begin{proof}
Clearly $\leq$ is reflexive. 
Showing that $\leq$ is antisymmetric amounts to showing that there are no distinct adjacent vertices of the same $s$-type. Indeed, if $v \sim_s w$ with $v \not= w$, then the edge between $v$ and $w$ 
is contained in a real simplex $\delta$, on which the projection to $S'_\cb$ is injective, so $\overline v$ and $\overline w$ are distinct and adjacent in $S'_\cb$. Adjacent vertices in the Coxeter complex cannot have the same type, so $s(v) = s(\overline v) \not= s(\overline w) = s(w)$. This means if $v \leq w$ and $w \leq v$, then $v \sim_s w$ and $s(v) = s(w)$, thus we must have $v = w$.
It remains to show that $\leq$ is transitive. 

Suppose $v,w,z$ are vertices such that $v \leq w$ and $w \leq z$. 
By Proposition~\ref{prop:comparable contained in real}, there is a real simplex $\delta$ containing $\{v,w,z\}$. Then $s(v) \leq s(w) \leq s(z)$, or more specifically, 
Lemma~\ref{lem:leq s containment} says that $\pos(v) \subseteq \pos(w) \subseteq \pos(z)$ and $\neg(v) \subseteq \neg(w) \subseteq \neg(z)$. Then 
Lemma~\ref{lem:adjacency in deligne}(1) implies $v \sim_s z$, and thus $v \leq_s z$, since $s(v) \leq s(z)$. 
\end{proof}

\begin{cor} \label{prop:A3 simplex chain}
    $\sd_{\ca(3)}'$ ($= \sd_{\cb}'$) is a flag complex.
\end{cor}

\begin{proof}
Suppose $V$ is a collection of vertices which are pairwise adjacent. By Proposition~\ref{prop:comparable contained in real}, there is a real simplex $\delta$ containing $V$.
Then $\overline \delta$ contains the vertices of $\overline V$ and in particular the edges between them. Since the vertices of $\overline V$ are pairwise adjacent and $S'_\cb$ is flag, they span a simplex $\overline{\delta_0}$ of $S'_\cb$. We know $\delta$ is isomorphic to $\overline \delta$, so we can pull back $\overline{\delta_0}$ to see that $V$ spans a simplex $\delta_0$ of $\sd_\cb'$.
\end{proof}

\subsection{Verifying the link condition} 

Now we want to reduce checking the bowtie free and upward flag conditions for the subdivided skewed $A_n$ type arrangements to only needing to verify certain cycles in the 1-skeleton of the subdivided Artin complex can be ``filled'' in a certain way. We start with Propositions \ref{prop:subdiv 4 cycles} and \ref{prop:subdiv 6 cycles}, which show how to fill these certain cycles, then conclude with Proposition~\ref{prop:contractible} to show how filling these cycles implies the full strength of bowtie free and upward flag. First, we will need a technical lemma to make some of the future arguments faster.

Define $(S_\cb')^+$ to be the full subcomplex of $S_\cb'$ on the non-negative vertices, and $S_\cb^-$ to be the full subcomplex on the non-positive vertices. Then let $(\Delta_\cb')^+$ and $(\Delta_\cb')^-$ be the pullback of $(S_\cb')^+$ and $(S_\cb')^-$ under the projection $\Delta_\cb' \to S_\cb'$, respectively. Recall that Proposition~\ref{prop:positive part} says that the restriction of the $s$-order to $(\Delta_\cb')^+$ is a partial order which is bowtie free and upward flag. The following proposition deals with $(\Delta_\cb')^-$.

\begin{lem} \label{lem:pos neg isom}
    There is a combinatorial isomorphism $\iota : \sd_\cb' \to \sd_\cb'$ which is $s$-order preserving, $u$-order reversing, and maps $(\sd_\cb')^+$ isomorphically to $(\sd_\cb')^-$ (and vice versa). In particular, the restriction of the $s$-order to $(\Delta_\cb')^-$ is a partial order which is bowtie free and upward flag.
\end{lem}

\begin{proof}
    Using the perspective of Definition \ref{def:falk complex of groups}, the existence of this isomorphism is immediate: the inversion map on the sphere induces an isomorphism of the complex of groups $\c U$, and extends to the development $\mathbb F_\ca$ with the claimed properties. We will give a more detailed proof using the perspective of Definition \ref{def:falk}.
    
    Consider $-(\cdot) : \bb[n]{R} \to \bb[n]{R}$, the usual inversion map $p \mapsto -p$ on \bb[n]{R}. This restricts to a combinatorial automorphism of $S_\cb$ which preserves the subdivision $S_\cb'$, and a combinatorial isomorphism of $\Sigma_\cb$.
    We can extend this to an automorphism of the Salvetti complex $\widehat \Sigma_\cb$ as follows. The map on the vertices is still given simply by the inversion map (as the vertices of $\widehat \Sigma_\cb$ are the vertices of $\Sigma_\cb$). Suppose $F$ is a face of $\Sigma_\cb$ and $v$ is a vertex of $F$, so $[F,v]$ is a face of $\widehat\Sigma_\cb$. Then we define $-[F,v] = [-F, -v]$. It is easy to see that this map gives a well-defined automorphism of $\widehat\Sigma_\cb$. 
    Let $\rho : \widetilde K \to \widehat\Sigma_\cb$ be the universal cover of $\widehat\Sigma_\cb$, and let $\kappa : \widetilde K \to \widetilde K$ be the lift of the inversion map $-(\cdot) : \widehat\Sigma_\cb \to \widehat\Sigma_\cb$ to $\widetilde K$. In particular, $\rho \circ \kappa = -\rho$.

    We now define $\iota : \sd_\cb \to \sd_\cb$ on the unsubdivided complex. Let $\c F$ be the collection of faces of $S_\cb$, and for $F \in \c F$, let $\Lambda_F$ be an index set of the elevations of the standard subcomplex $\widehat F$ (cf.~Definition~\ref{def:falk} and Section~\ref{subsec:link}). 
    First we define $\iota_0$ as an automorphism of $\bigsqcup_{F \in \c F} F \times \Lambda_F$. 
    For every $F \in \c F$, there is a bijection $\Lambda_F \to \Lambda_{-F}$, which we will denote $\lambda \mapsto -\lambda$, induced by $\kappa$. More specifically, since $\kappa$ commutes with the projection, it bijectively maps the elevations of $F$ to the elevations of $-F$ in such a way that preserves nesting of elevations, and this induces the map on the index sets. 
    Then we define $\iota_0(F \times \{\lambda\}) = (-F) \times \{-\lambda\}$. By our definition of the bijection $\Lambda_F \to \Lambda_{-F}$, it is clear that this respects the gluing of the Falk complex, and thus descends to an automorphism $\iota$ on $\left(\bigsqcup_{F \in \c F} F \times \Lambda_F\right)/\sim = \Delta_\cb$. Letting $\overline{\,\cdot\,} : \sd_\cb \to S_\cb$ be the usual projection, it is immediate from the definitions that $\overline {\iota(x)} = - \overline x$ for all $x \in \sd_\cb$.  In particular, $\iota$ preserves the subdivision $\sd_\cb'$. 

    Thus we see for a vertex $v \in \Delta_\cb'$ that $s(v) = s(\overline v) = s(-\overline v) = s(\overline{\iota(v)}) = s(\iota(v))$, but for a vertex $v \in \Delta_\cb$ that $u(v) = u(\overline v) = n+1 - u(-\overline v) = n+1 - u(\overline{\iota(v)}) = n+1 - u(\iota(v))$.
    So if $v,w \in (\Delta_\cb')^0$ and $v \leq_s w$, then $\iota(v) \sim_s \iota(w)$ since $\iota$ is an automorphism, and $s(\iota(v)) = s(v) \leq s(w) = s(\iota(w))$, so $\iota(v) \leq_s \iota(w)$. So $\iota$ is $s$-order preserving. But if $v,w \in (\Delta_\cb)^0$ and $v \leq_u w$, then $\iota(v) \sim_u \iota(w)$ since $\iota$ is an automorphism, while $u(\iota(v)) = n+1 - u(v) \geq n+1 - u(w) = u(\iota(w))$, so $w \leq_u v$. So $\iota$ is $u$-order reversing.
    In particular, it is an order-preserving isomorphism of $( ((\Delta_\cb')^+)^0, <_s)$ and $( ((\Delta_\cb')^-)^0, <_s)$, so $(\Delta_\cb')^-$ inherits all the $s$-order properties enjoyed by $(\Delta_\cb')^+$. 
\end{proof}

\begin{prop} \label{prop:subdiv 4 cycles}
    Any embedded $4$-cycle in $\sd'_\cb$ of $s$-type $1n1n$ has a central vertex, i.e., a vertex which is $s$-adjacent or equal to each vertex of the cycle. 
\end{prop}

\begin{proof}
    Let $(v_1,w_1,v_2,w_2)$ be an embedded $4$-cycle in $\sd'_\cb$ with $s(v_i) = 1$ and $s(w_i) = n$ for $i = 1,2$. 
    For $i = 1,2$, since $\#\pos(v_i) + \#\neg(v_i) = s(v_i) = 1$, this means either $\pos(v_i) = \varnothing$ or $\neg(v_i) = \varnothing$. 
    In particular, the $v_i$ must be real vertices.

    First, assume $\neg(v_1) = \neg(v_2) = \varnothing$. %
    Choose some $i = 1,2$. If $w_i$ is real, then since it's $s$-adjacent to the non-negative vertices $v_1$ and $v_2$, it is also non-negative (Lemma~\ref{lem:adjacency in deligne}(3)), and in particular, is $u$-adjacent to these vertices (Lemma~\ref{lem:adjacency in deligne}(2)). If $w_i$ is not real, then $w_i^+$ is non-negative (by definition) and $v_j \leq_u w_i^+$ by Lemma~\ref{lem:fake projections are adjacent}. We can summarize this as follows: for $i = 1,2$, if $w_i$ is not fake, define $w_i' = w_i$, and if $w_i$ is fake, define $w_i' = w_i^+$.
    Then $w_i'$ is non-negative and $v_i \leq_u w_j'$ for $i,j = 1,2$. Then by Proposition~\ref{prop:positive part}, there is a vertex $v \in (\sd_\cb')^+$ with $v_i \leq_u v$ and $v \leq_u w_i'$. See Figure \ref{fig:4cycle1}.
    \begin{figure}[!ht]
        \centering
        \begin{minipage}{0.49\textwidth}
        \includegraphics[width=0.9\linewidth]{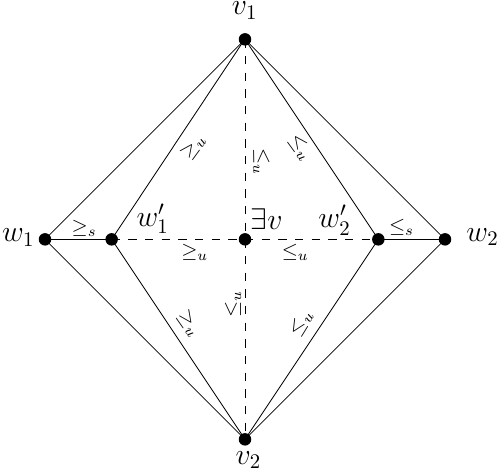}
        \caption{Filling the 4-cycle in the $u$-order}
        \label{fig:4cycle1}
        \end{minipage}
        \begin{minipage}{0.49\textwidth}
        \centering
        \includegraphics[width=0.9\linewidth]{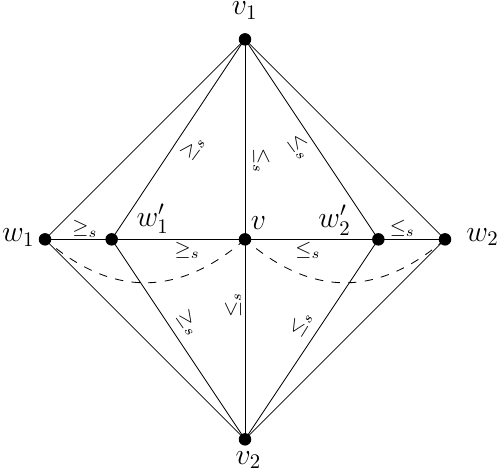}
        \caption{Transferring to the $s$-order}
        \label{fig:4cycle2}
        \end{minipage}
    \end{figure}
    In particular, $v_i \leq_s v$ and $v \leq_s w_i'$ by Proposition~\ref{prop:basic artin}(5). If $w_i$ is real then $v \leq_s w_i' = w_i$. If it is fake, then $w_i' = w_i^+ \leq_s w_i$ by Lemma~\ref{lem:fake not less than real}. Since the $s$-order is transitive, $v \leq_s w_i$. See Figure \ref{fig:4cycle2}.

    Next, assume $\pos(v_1) = \pos(v_2) = \varnothing$. This case follows immediately from Lemma~\ref{lem:pos neg isom}, but we provide the details for completeness.
    Choose some $i = 1,2$. If $w_i$ is real, then since it's $s$-adjacent to the non-positive vertices $v_1$ and $v_2$, it is also non-positive (Lemma~\ref{lem:adjacency in deligne}(3)). Note we have already assumed that $v_j \leq_s w_i$, so $\iota(v_j) \leq_s \iota(w_i)$. If $w_i$ is not real, then $w_i^-$ is non-positive (by definition) and $w_i^- \leq_u v_j$ by Lemma~\ref{lem:fake projections are adjacent}. But then since these vertices live in $(\Delta_\cb')^-$, we know from Lemma~\ref{lem:pos neg isom} that $\iota(v_j) \leq_u \iota(w_i^-)$ since $\iota$ is $u$-order reversing, and consequently $\iota(v_j) \leq_s \iota(w_i^-)$ since their images under $\iota$ are non-negative, and the $u$- and $s$-orders agree on the non-negative part.
    We can summarize this as follows: for $i = 1,2$, if $w_i$ is not fake, define $w_i' = w_i$, and if $w_i$ is fake, define $w_i' = w_i^-$.
    Then $w_i'$ is non-negative and $\iota(v_i) \leq_s \iota(w_j')$ for $i,j = 1,2$. Then by Proposition~\ref{prop:positive part}, there is a vertex $v \in (\sd_\cb')^+$ with $\iota(v_i) \leq_s v$ and $v \leq_s \iota(w_i')$. In particular $v_i \leq_s \iota^{-1}(v)$ and $\iota^{-1}(v) \leq_s w_i'$ since $\iota$ is $s$-order preserving. If $w_i$ is real then $\iota^{-1}(v) \leq_s w_i' = w_i$. If it is fake, then $w_i' = w_i^- \leq_s w_i$ by Lemma~\ref{lem:fake not less than real}. Since the $s$-order is transitive, $\iota^{-1}(v) \leq_s w_i$.

    Now assume $\neg(v_1) = \pos(v_2) = \varnothing$ (the case $\pos(v_1) = \neg(v_2) = \varnothing$ is identical). Note that for $i = 1,2$, $w_i$ cannot be a real vertex: if it was non-negative, then it could not be $s$-adjacent to the non-positive vertex $v_2$, and if it were non-positive, it could not be $s$-adjacent to the non-negative vertex $v_1$ (Lemma~\ref{lem:adjacency in deligne}(3)). So both $w_1$ and $w_2$ are fake vertices. %

    We claim that $v_1 \leq_u v_2$. Indeed, since $v_1$ and $v_2$ are $s$-adjacent to the fake vertex $w_1$, then $v_1 \leq_u w_1^+$ since $v_1$ is non-negative, and $w_1^+ \leq_u v_2$ since $v_2$ is non-positive (Lemma~\ref{lem:fake projections are adjacent}). By transitivity of the $u$-order, $v_1 \leq_u v_2$, as claimed. See Figure \ref{fig:4cycle3}.
    \begin{figure}[!ht]
        \centering
        \begin{minipage}{0.49\textwidth}
        \includegraphics[width=0.9\linewidth]{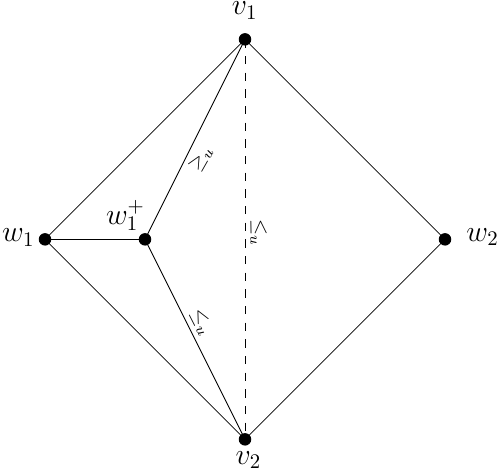}
        \caption{$v_1 \leq_u v_2$}
        \label{fig:4cycle3}
        \end{minipage}
        \begin{minipage}{0.49\textwidth}
        \centering
        \includegraphics[width=0.9\linewidth]{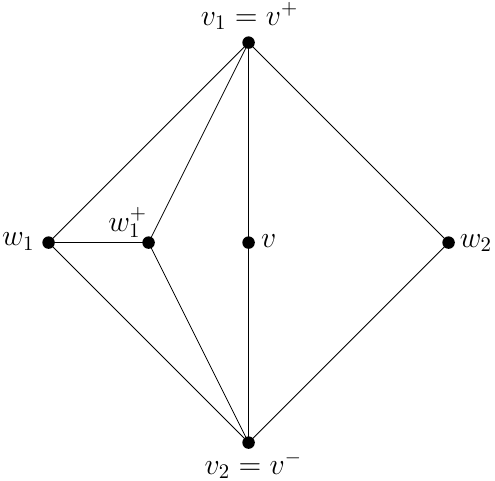}
        \caption{$v$ and its real projections}
        \label{fig:4cycle4}
        \end{minipage}
    \end{figure}
    Then by Lemma~\ref{lem:adjacency in deligne}(3), there is a (fake) vertex $v$ (possibly equal to either $w_1$ or $w_2$) which lies on the real edge between $v_1$ and $v_2$. Then 
    Lemma~\ref{lem:fake projects to real} implies that $v^+ = v_1$ and $v^- = v_2$ (the real edge that $v$ lies on is clearly unique); in particular, $v \sim_s v_i$ for $i = 1,2$. See Figure \ref{fig:4cycle4}.
    We have already seen that $v^+ = v_1 \leq_u w_1^+$. For identical reasons (namely, Lemma~\ref{lem:fake projections are adjacent}), we also have $v^+ = v_1 \leq_u w_2^+$, but also $w_1^- \leq_u v_2 = v^-$ and $w_2^- \leq_u v_2 = v^-$. Note that it is always the case that $w_1^+ \leq_u w_1^-$ and $w_2^+ \leq_u w_2^-$ (Lemma~\ref{lem:adjacency in deligne}(3)). So in summary, for $i = 1,2$, $v^+ \leq_u w_i^+ \leq_u w_i^- \leq_u v^-$. Thus by Lemma \ref{lem:chain of fake}, $v\sim_s w_i$ for $i = 1,2$.
    \begin{figure}[!ht]
        \centering
        \begin{minipage}{0.49\textwidth}
        \includegraphics[width=0.9\linewidth]{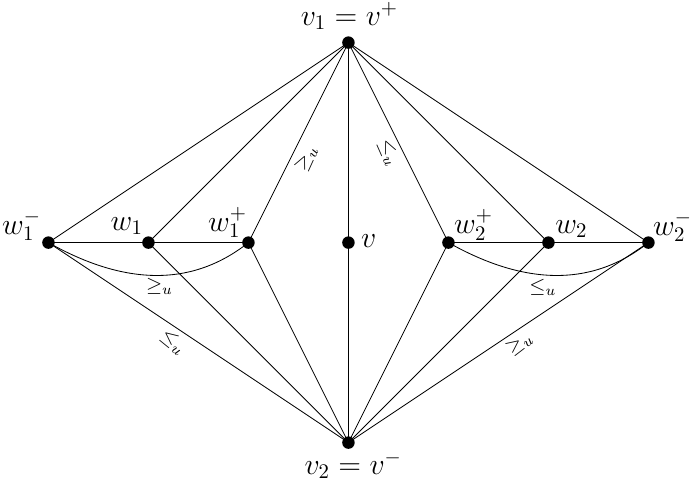}
        \caption{Known $u$-order inequalities}
        \label{fig:4cycle5}
        \end{minipage}
        \begin{minipage}{0.49\textwidth}
        \centering
        \includegraphics[width=0.9\linewidth]{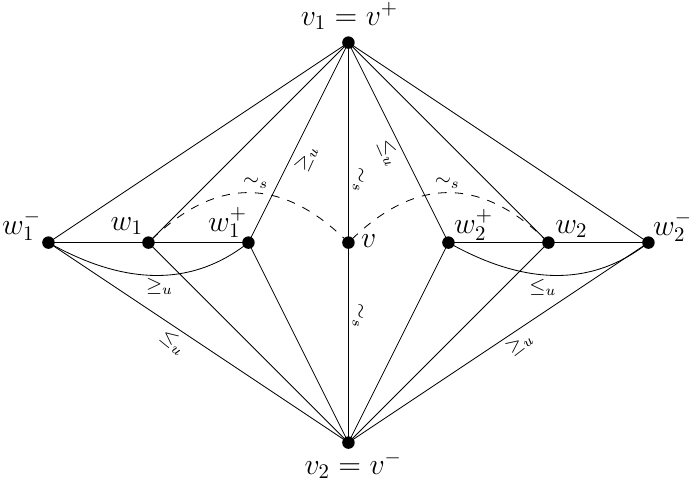}
        \caption{Applying Lemma \ref{lem:chain of fake}}
        \label{fig:4cycle6}
        \end{minipage}
    \end{figure}
\end{proof}

\begin{prop}\label{prop:subdiv 6 cycles}
    Any embedded $6$-cycle in $\sd'_\cb$ of $s$-type $1n1n1n$ has a vertex in $\sd'_\cb$ $s$-adjacent to all the $s$-type $1$ vertices. 
\end{prop}

\begin{proof}
    Let $(v_1,w_1,v_2,w_2,v_3,w_3)$ be an embedded $6$-cycle in $\sd'_\cb$ with $s(v_i) = 1$ and $s(w_i) = n$ for $i = 1,2,3$. For each $i$, 
    since $s(v_i) = 1$, then either $\pos(v_i) = \varnothing$ or $\neg(v_i) = \varnothing$, so the $v_i$ must be real vertices. We will take cases on which $v_i$ are non-positive and which are non-negative.

    Suppose first that all $v_i$ are non-negative. This means $1 = s(v_i) = u(v_i) = \#\pos(v_i)$ for each $i$. 
	For $i = 1,2,3$, if $w_i$ is real, define $w_i' = w_i$, and if $w_i$ is fake, define $w_i' = w_i^+$.
    By Lemma~\ref{lem:fake projections are adjacent}, if $w_i$ is fake and $v_j \leq_s w_i$ then $v_j \leq_u w_i^+$, and since $v_j$ and $w_i^+$ are non-negative, $v_j \leq_s w_i^+$. So regardless if $w_i$ is fake or not, we always have $v_j \leq_s w_i'$ whenever $v_j \leq_s w_i$. This means the non-negative $v_j$ are pairwise $s$-upper bounded by the non-negative $w_i'$, and Proposition~\ref{prop:positive part} tells us they have a common upper bound in the subcomplex $(\sd_\cb')^+$, hence in $\sd_\cb'$.

    \begin{figure}[!ht]
        \centering
        \begin{multicols}{3}
        \includegraphics[width=0.9\linewidth]{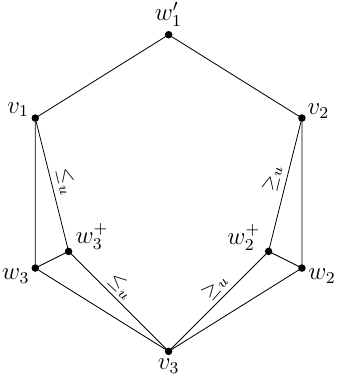}
        
        (a)
        
        \includegraphics[width=0.9\linewidth]{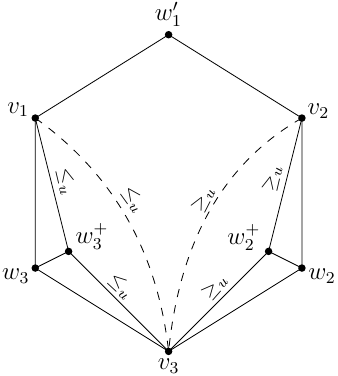}
        
        (b)
        
        \includegraphics[width=0.9\linewidth]{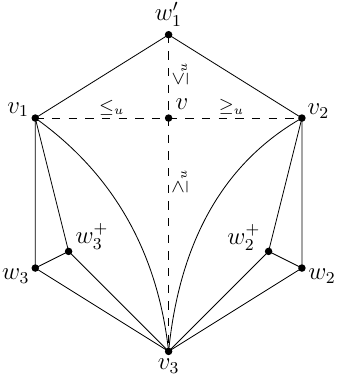}
        
        (c)
        \end{multicols}
        \caption{Filling vertices and edges in the $u$-order}
        \label{fig:6cyclea}
    \end{figure}
    
    Suppose exactly one $v_i$ is non-positive and the rest are non-negative. We may assume that $v_3$ is non-positive. 
    By a similar argument to the one given in the proof of Proposition~\ref{prop:subdiv 4 cycles}, $w_2$ and $w_3$ are fake vertices. We define $w_1'$ similarly to before: if $w_1$ is real, define $w_1' = w_1$, and if it's fake, define $w_1' = w_1^+$. 
    Notice that $1 = s(v_i) = u(v_i) = \#\pos(v_i)$ for $i = 1,2$, and $u(v_3) = n+1 - s(v_3) = n$. 
    Then $v_3$ is $u$-adjacent to both $v_1$ and $v_2$ in $\sd_\cb$:
    since $v_2$ is non-negative, we have $v_2 \leq_u w_2^+$, and since $v_3$ is non-positive, we have $w_2^+ \leq_u v_3$ (Lemma~\ref{lem:fake projections are adjacent}). See Figure \ref{fig:6cyclea}(a). By transitivity of the $u$-order, $v_2 \leq_u v_3$ (Figure \ref{fig:6cyclea}(b)).  The argument for $v_1$ is identical, replacing $w_2^+$ with $w_3^+$.
    Now the $4$-cycle $(v_1,w_1',v_2,v_3)$ in $\sd_\cb$ is a bowtie, so there is a central (real) vertex $v$ which is $u$-adjacent to each vertex, or more specifically, $v_1 \leq_u v$, $v_2 \leq_u v$, $v \leq_u w_1'$, and $v \leq_u v_3$. In particular, $1 < u(v) < n$. See Figure \ref{fig:6cyclea}(c).
    There are two possibilities.

    \begin{figure}[!ht]
        \centering
        \begin{minipage}{0.49\textwidth}
        \includegraphics[width=0.9\linewidth]{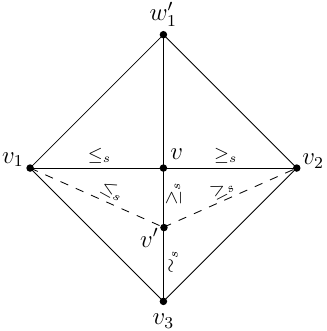}
        \caption{$v$ non-negative}
        \label{fig:6cycleb}
        \end{minipage}
        \begin{minipage}{0.49\textwidth}
        \centering
        \includegraphics[width=0.9\linewidth]{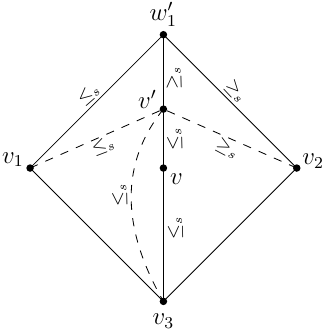}
        \caption{$v$ non-positive}
        \label{fig:6cyclec}
        \end{minipage}
    \end{figure}
    
    Suppose $v$ is non-negative. Then $v$ is adjacent to $v_1$ and $v_2$ in $\sd_\cb'$ and since these vertices are non-negative, $v_1 \leq_s v$ and $v_2 \leq_s v$. But there is a fake vertex $v'$ on the edge between $v$ and $v_3$. (In particular, $v' \sim_s v_3$.) Then $v \leq_s v'$ by Lemma~\ref{lem:fake not less than real}. By transitivity of the $s$-order, $v_1 \leq_s v'$ and $v_2 \leq_s v'$, so $v'$ is adjacent to each of $v_1$, $v_2$, and $v_3$. See Figure \ref{fig:6cycleb}.

    Suppose $v$ is non-positive. Since it is $u$-adjacent to the non-negative vertex $w_1'$, there is a fake vertex $v'$ on the real edge between $v$ and $w_1'$ (Lemma~\ref{lem:adjacency in deligne}(3)). In particular, $w_1' \leq_s v'$ and $v \leq_s v'$. For $i = 1,2$, since $v_i \leq_s w_1'$, we have $v_i \leq v'$. Last, $v \leq_u v_3$, but since these vertices are non-positive, $v_3 \leq_s v$. Then by transitivity, $v_3 \leq_s v'$. See Figure \ref{fig:6cyclec}.

    The cases with exactly two $v_i$ non-positive and with all $v_i$ non-positive are identical to these cases after applying Lemma~\ref{lem:pos neg isom}.
\end{proof}

\begin{thm}
\label{prop:contractible}
$((\sd'_{\ca(3)})^0, <)$ is bowtie-free and upward flag.
\end{thm}

\begin{proof}
In order to make the dimension clear, we let $\cb_k$ denote the arrangement $\ca(3)$ in \bb[k]{R}.
We start by showing $((\sd'_{\cb_k})^0, <)$ is bowtie-free by induction on $k$. 
When $k = 1$, there is nothing to show. %
So assume $k \geq 2$ and $((\sd'_{\cb_j})^0, <)$ is bowtie-free for all $j < k$. 

\begin{figure}[!ht]
    \centering
    \begin{multicols}{3}
        \includegraphics[width=\linewidth]{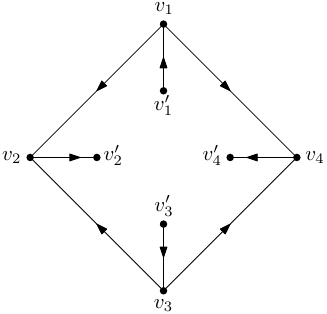} 

        (a)
        
        \includegraphics[width=\linewidth]{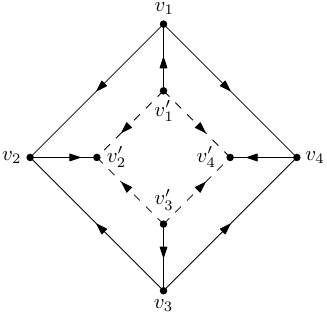}

        (b)
        
        \includegraphics[width=\linewidth]{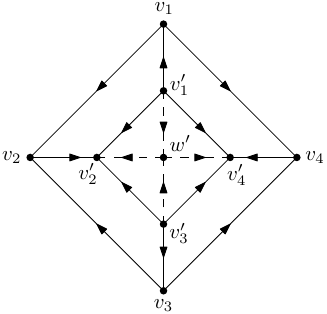}

        (c) 
    \end{multicols}
    
    \caption{Finding $w'$}
    \label{fig:bowtie1}
\end{figure}

Let $\{v_1,v_2,v_3,v_4\}$ be a set of vertices in $\sd'_{\cb_k}$ with 
$v_1 <_s v_2$, $v_2 >_s v_3$, $v_3 <_s v_4$, and $v_4 >_s v_1$. Let $\gamma = (v_1,v_2,v_3,v_4)$. If $\gamma$ is not embedded, then the vertices are not pairwise distinct, so $\gamma$ has a central vertex (i.e., a vertex which is $s$-adjacent or equal to each vertex of the cycle).
So suppose $\gamma$ is embedded.
For $i = 1,3$, choose a vertex $v_i' \simeq_s v_i$ with $s(v_i') = 1$. Similarly, for $i = 2,4$, choose a vertex $v_i' \simeq_s v_i$ with $s(v_i') = k$. See Figure \ref{fig:bowtie1}(a); throughout this proof, an arrow pointing from a vertex $x$ to a vertex $y$ means $x \leq_s y$. Since $\gamma$ is embedded, we can choose such vertices so that $v_1' \not= v_3'$ and $v_2' \not= v_4'$. 
By transitivity of the $s$-order, 
$v_1' <_s v_2'$, $v_2' >_s v_3'$, $v_3' <_s v_4'$, and $v_4' >_s v_1'$.  See Figure \ref{fig:bowtie1}(b).
In particular, $\gamma' = (v_1', v_2', v_3', v_4')$ is an embedded 4-cycle of $s$-type $1k1k$. Proposition~\ref{prop:subdiv 4 cycles} guarantees the existence of a central vertex $w'$ for $\gamma'$, as in Figure \ref{fig:bowtie1}(c). We will now show how $w'$ gives rise to a central vertex of $\gamma$, showing that $\{v_1,v_2,v_3,v_4\}$ is not a bowtie.

\begin{figure}[!ht]
    \centering

    \includegraphics[width=0.35\linewidth]{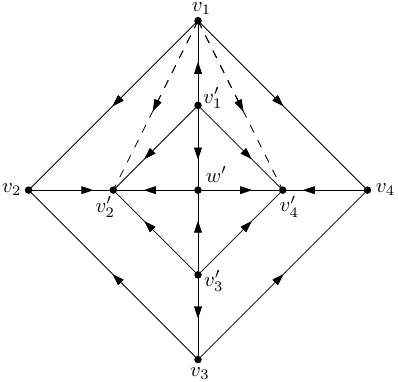} 
    \caption{Applying transitivity of $\leq_s$}
    \label{fig:bowtie2-0}
    \vspace{3ex}
\end{figure}
    
\begin{figure}[!ht]
    \centering
    \begin{multicols}{3}
        
        \includegraphics[width=0.95\linewidth]{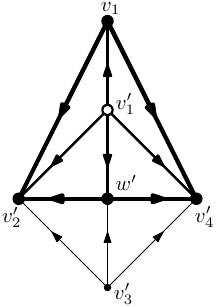} 

        (a)

        \vspace{1em}
        \includegraphics[width=0.95\linewidth]{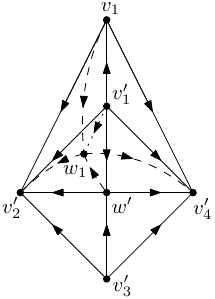}

        (b)

        \vspace{3ex}
        
        \includegraphics[width=0.95\linewidth]{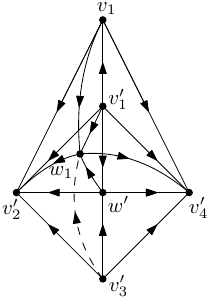}

        (c)
        
    \end{multicols}

    \caption{Finding $w_1$}
    \label{fig:bowtie2}
\end{figure}

If $v_1' = v_1$, define $w_1 = w'$. If not, then note that $v_1 <_s v_i \leq_s v_i'$ for $i = 2,4$. It follows from transitivity that $v_1 <_s v_i'$ for $i = 2,4$, as in Figure \ref{fig:bowtie2-0}. So, in $\lk(v_1', \sd'_{\cb_k})$, there is a $4$-cycle $\gamma_1 = (v_1, v_2', w', v_4')$, seen in Figure \ref{fig:bowtie2}(a). Since $s(v_1') = 1$, we know $v_1'$ is a real vertex, and $\lk(v_1', \sd'_{\cb_k}) \cong \sd'_{\cb_{k-1}}$. By induction, this complex is bowtie free, so $\gamma_1$ has a central vertex in $\lk(v_1', \sd'_{\cb_k})$, which we will call $w_1$. See Figure \ref{fig:bowtie2}(b). Then note that $v_3' \leq_s w'$ and $w' \leq w_1$, so by transitivity, $v_3' \leq w_1$, as seen in Figure \ref{fig:bowtie2}(c).  In either case, we can simplify the diagram by ignoring $w'$ and $v_1'$ to obtain the diagram in Figure 
\ref{fig:bowtie3}.

\begin{figure}[!ht]
    \centering
    \includegraphics[width=0.35\linewidth]{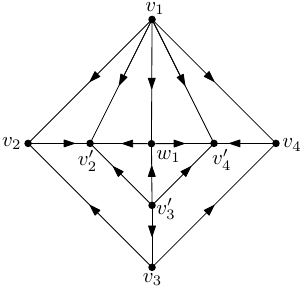}
    \caption{The first simplified diagram}
    \label{fig:bowtie3}
\end{figure}

\begin{figure}[!ht]
    \centering
    \begin{multicols}{3}

        \includegraphics[width=0.95\linewidth]{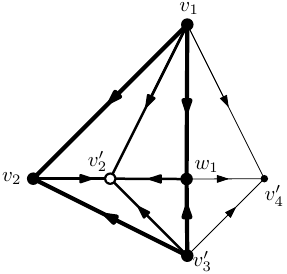}

        (a)

        \includegraphics[width=0.95\linewidth]{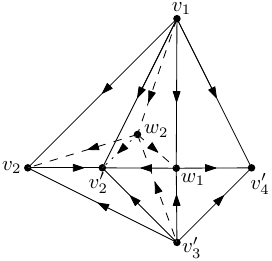} 

        (b)
        
        \vspace{3ex}
        
        \includegraphics[width=0.95\linewidth]{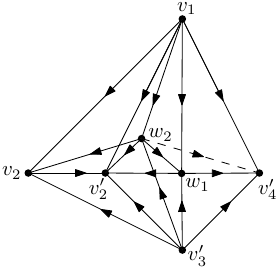}

        (c)
        
    \end{multicols}
    
    \caption{Finding $w_2$}
    \label{fig:bowtie4}
\end{figure}

If $v_2' = v_2$, define $w_2 = w_1$. If not, then, for similar reasons as before, $v_3' <_s v_2$. So, in $\lk(v_2', \sd'_{\cb_k})$, there is a $4$-cycle $\gamma_2 = (v_1, v_2, v_3', w_1)$, illustrated in Figure \ref{fig:bowtie4}(a). We claim that $\lk(v_2', \sd'_{\cb_k})$ is bowtie-free, hence this cycle has a central vertex.

Since $s(v_2') = k$, it is either a real vertex or lies on the midpoint of the real line between $(v_2')^+$ and $(v_2')^-$. By Lemma~\ref{lem:link}, if $v_2'$ is a real vertex, then $\lk(v_2', \sd_{\cb_k})$ (the link in the unsubdivided complex) is isometric to either $\sd_{\cb_{k-1}}$ or a spherical join $\sd_{\cb_i} \circ \sd_{\cb_j}$ where $i + j = k-1$ and $i,j > 0$. In particular,  $\lk(v_2', \sd_{\cb_k}')$ (the link in the subdivided complex) is isometric to the $s$-subdivision of either of these complexes. 
If $v_2'$ is a fake vertex, then it is the midpoint of a real edge $e$ of $\sd_{\cb_k}$. Again by Lemma~\ref{lem:link}, $\lk(e,\sd_{\cb_k})$ is isometric to either $\sd_{\cb_{k-2}}$, or $\sd_{\cb_{i}} \circ \sd_{\cb_{j}}$ where $i + j = k-2$, or $\sd_{\cb_{i}} \circ \sd_{\cb_{j}} \circ \sd_{\cb_{\ell}}$ where $i + j + \ell = k-2$ (and each are positive). It follows that $\lk(v_2', \sd_{\cb_k}')$ is isometric to the $s$-subdivision of this link. 
We can summarize this by saying that, regardless if $v_2'$ is real or fake, $\lk(v_2', \sd'_{\cb_k})$ is isometric to the $s$-subdivision of the join $\sd_{\cb_{i}} \circ \sd_{\cb_{j}} \circ \sd_{\cb_{\ell}}$ where each $0 \leq i,j,\ell < k$ with at least one of $i$, $j$, and $\ell$ positive, and if any are 0, we exclude the corresponding complex from the join. By induction, each of these join factors are bowtie free, so it follows from Lemma~\ref{lem:decomposing 4} and its proof that their $s$-subdivided join $\lk(v_2', \sd'_{\cb_k})$ is bowtie free.

So, since this link is bowtie-free, $\gamma_2$ has a central vertex in $\lk(v_2', \sd'_{\cb_k})$, which we will call $w_2$. See Figure \ref{fig:bowtie4}(b). Then $w_2 \leq_s w_1 \leq_s v_4'$, so $w_2 \leq_s v_4'$ by transitivity (Figure \ref{fig:bowtie4}(c)). In either case, we can simplify the diagram by ignoring $w_1$ and $v_2'$ to obtain the diagram in Figure \ref{fig:bowtie5}.

\begin{figure}[!ht]
    \centering
    \includegraphics[width=0.35\linewidth]{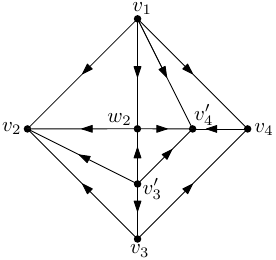}
    \caption{The second simplified diagram}
    \label{fig:bowtie5}
\end{figure}

\begin{figure}[!ht]
    \centering
    \begin{multicols}{3}

        \includegraphics[width=0.95\linewidth]{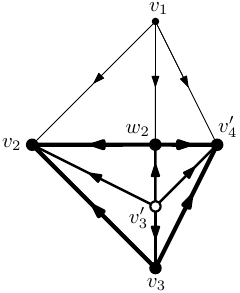}

        (a)

        \includegraphics[width=0.95\linewidth]{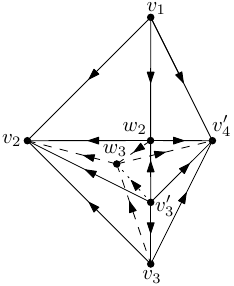} 

        (b)
        
        \vspace{3ex}
        
        \includegraphics[width=0.95\linewidth]{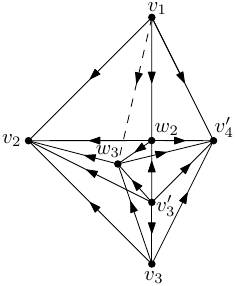}

        (c)
        
    \end{multicols}
    
    \caption{Finding $w_3$}
    \label{fig:bowtie6}
\end{figure}

If $v_3' = v_3$, define $w_3 = w_2$. If not, then in $\lk(v_3', \sd'_{\cb_k})$, there is a $4$-cycle $(w_2, v_2, v_3, v_4')$ (Figure \ref{fig:bowtie6}(a)). By an argument similar to the one given for $\gamma_1$, this cycle has a central vertex in $\lk(v_3', \sd'_{\cb_k})$, which we will call $w_3$ (Figure \ref{fig:bowtie6}(b)). Then $v_1 \leq_s w_2$ and $w_2 \leq_s w_3$, so $v_1 \leq_s w_3$ (Figure \ref{fig:bowtie6}(c)). We can then reduce the diagram to Figure \ref{fig:bowtie7}.

\begin{figure}[!ht]
    \centering
    \includegraphics[width=0.35\linewidth]{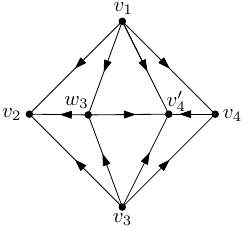}
    \caption{The third simplified diagram}
    \label{fig:bowtie7}
\end{figure}

\begin{figure}[!ht]
    \centering
    \begin{multicols}{3}
        \includegraphics[width=0.95\linewidth]{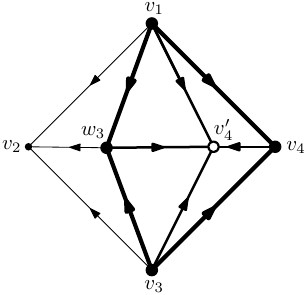} 

        (a)

        \includegraphics[width=0.95\linewidth]{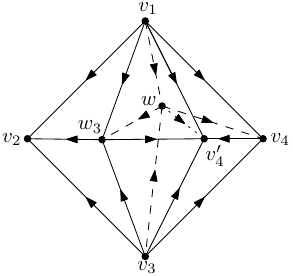}

        (b)

        \includegraphics[width=0.95\linewidth]{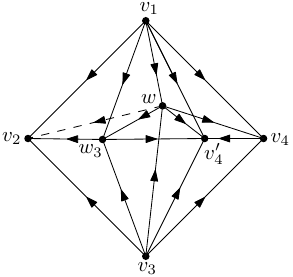} 

        (c)
        
    \end{multicols}
    
    \caption{Finding $w$}
    \label{fig:bowtie8}
\end{figure}

If $v_4' = v_4$, define $w = w_3$. If not, then in $\lk(v_4', \sd'_{\cb_k})$, there is a $4$-cycle $(v_1, w_3, v_3, v_4)$ (Figure \ref{fig:bowtie8}(a)). By an argument similar to the one given for $\gamma_2$, this cycle has a central vertex in $\lk(v_4', \sd'_{\cb_k})$, which we will call $w$ (Figure \ref{fig:bowtie8}(b)). 
Then $w \leq_s w_3 \leq_s v_2$, so $w \leq_s v_2$ (Figure \ref{fig:bowtie8}(c)).
Thus $w$ is a central vertex for $\gamma$, so $\sd'_{\cb_k}$ is bowtie-free (Figure \ref{fig:bowtie9}).

\begin{figure}[!ht]
    \centering
    \includegraphics[width=0.3\linewidth]{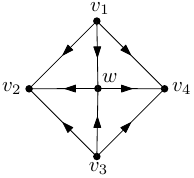}
    \caption{The final diagram}
    \label{fig:bowtie9}
\end{figure}

Next we show $((\sd'_{\cb_k})^0, <)$ is upward flag by induction on $k$. If $k = 1$ then there is nothing to show.
Assume $k \geq 2$ and $((\sd'_{\cb_j})^0, <_s)$ is upward flag for all $j < k$. 
Suppose $\{w_1,w_2,w_3\}$ is a collection of pairwise $s$-upper bounded vertices of $\sd'_{\cb_k}$, with $z_i$ an upper bound for $\{w_i, w_{i+1}\}$ for $i = 1,2,3$ (with indices taken cyclically). See Figure \ref{fig:flag1}(a). For each $i = 1,2,3$, we may assume that $s(z_i) = k$ (for example, by replacing $z_i$ with any vertex $v$ of $s$-type $k$ with $z_i \leq_s v$). We will repeat a procedure similar to the one used for $4$-cycles to determine an upper bound of the $w_i$.

\begin{figure}[!ht]
    \centering 

    \begin{multicols}{4}
        
    \includegraphics[width=0.95\linewidth]{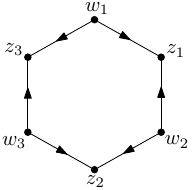}

    (a)
    
    \includegraphics[width=0.95\linewidth]{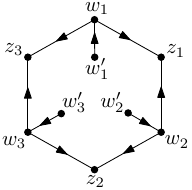}

    (b)

    \includegraphics[width=0.95\linewidth]{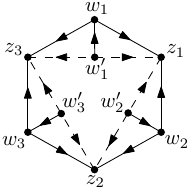}

    (c)

    \includegraphics[width=0.95\linewidth]{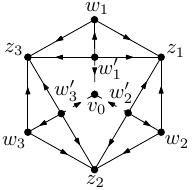}

    (d)
    \end{multicols}
    
    \caption{Finding $v_0$}
    \label{fig:flag1}
\end{figure}

For each $i = 1,2,3$, let $w_i'$ be a vertex with $w_i' \leq_s w_i$ and $s(w_i') = 1$.  
(Note that if $s(w_i) = 1$, this implies $w_i' = w_i$.)  See Figure \ref{fig:flag1}(b). 
Then $\{w_1',w_2',w_3'\}$ is a collection of $s$-type $1$ vertices pairwise upper bounded by $s$-type $k$ vertices, giving a 6-cycle $\beta = (w_1', z_1, w_2', z_2, w_3', z_3)$ of $s$-type $1k1k1k$.  See Figure \ref{fig:flag1}(c). If $\beta$ is not embedded, then either the $w_i'$ are not pairwise distinct or the $z_i$ are not pairwise distinct.
If the $w_i'$ cannot be chosen to be pairwise distinct, then the $w_i$ could not have been pairwise distinct to start with, so one of the $z_i$ will be an upper bound of each vertex and we are done. If the $z_i$ are not pairwise distinct, then one of them is an upper bound for each of the $w_i$ and we are done. So we may assume that the $6$-cycle $\beta$ is embedded. By Proposition~\ref{prop:subdiv 6 cycles}, there exists a vertex $v_0$ which is an upper bound of the $w_i'$.  See Figure \ref{fig:flag1}(d).  We now find a vertex $v$ which is an upper bound of the $w_i$.

\begin{figure}[!ht]
    \centering 

    \begin{multicols}{4}
        
    \includegraphics[width=0.95\linewidth]{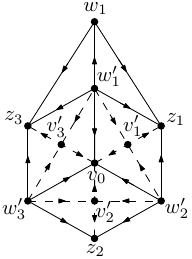}

    (a)
    
    \includegraphics[width=0.95\linewidth]{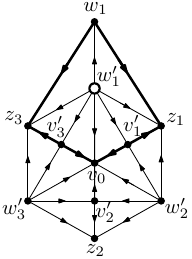}

    (b)

    \includegraphics[width=0.95\linewidth]{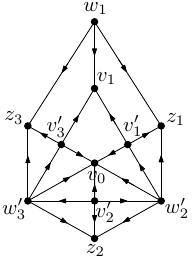}

    (c)

    \includegraphics[width=0.95\linewidth]{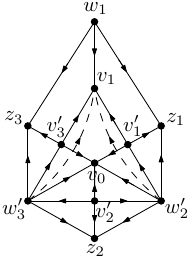}

    (d)
    \end{multicols}
    
    \caption{Finding $v_1$}
    \label{fig:flag2}
\end{figure}

If $s(w_1) = 1$ (so $w_1' = w_1$), then let $v_1 = v_0$. Otherwise, define $v_1$ as follows.
For each $i$, consider the 4-cycle $\gamma_i^{(1)} = (w_i', v_0,w_{i+1}', z_i)$. Since $((\sd'_{\cb_k})^0, <_s)$ is bowtie-free by our above work, there must be some vertex $v_i'$ such that $w_i' \leq_s v_i'$, $w_{i+1}' \leq_s v_i'$, $v_i' \leq_s v_0$, and $v_i' \leq_s z_i$. See Figure \ref{fig:flag2}(a). Now inside $\lk(w_1',\sd'_{\cb_k})$, the vertices $\{ w_1, v_1', v_3' \}$ are pairwise $s$-upper bounded: $z_3$ is an upper bound for $\{v_{3}',w_1\}$, $z_1$ is an upper bound for $\{w_1,v_1\}$, and $v_0$ is an upper bound for $\{v_1',v_3'\}$. See Figure \ref{fig:flag2}(b). Similar to the bowtie free case previously, this link is the $s$-subdivision of the join of posets, each of which are upward flag by induction, and thus is itself upward flag. So, there is an $s$-upper bound $v_1$ for  $\{ w_1, v_1', v_3' \}$ in the link of $w_1'$. See Figure \ref{fig:flag2}(c).  In particular, $w_2' \leq_s v_1' \leq_s v_1$ and $w_3' \leq_s v_3' \leq_s v_1$, so $v_1$ is an upper bound of $\{w_1,w_2',w_3'\}$.  See Figure \ref{fig:flag2}(d).

\begin{figure}[!ht]
    \centering 

    \begin{multicols}{2}
        
    \includegraphics[width=0.75\linewidth]{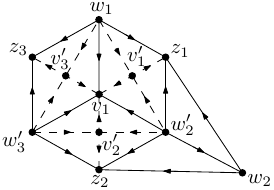}

    (a)
    
    \includegraphics[width=0.75\linewidth]{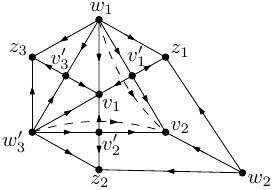}

    (b)
    
    \end{multicols}
    
    \caption{Finding $v_2$}
    \label{fig:flag3}
\end{figure}

We now repeat a similar procedure on the vertices $\{w_1,w_2',w_3'\}$. If $s(w_2) = 1$, define $v_2 = v_1$; otherwise, define $v_2$ as follows. 
Let $\gamma_1^{(2)} = (w_1, v_1,w_{2}', z_1)$, 
$\gamma_2^{(2)} = (w_2', v_1,w_{3}', z_2)$, and
$\gamma_3^{(2)} = (w_3', v_1,w_1, z_3)$.
To simplify notation, 
let $v_1'$ now denote a vertex such that $w_1 \leq_s v_1'$, $w_{2}' \leq_s v_1'$, $v_1' \leq_s v_1$, and $v_1' \leq_s z_1$, 
let $v_2'$ now denote a vertex such that $w_2' \leq_s v_2'$, $w_{3}' \leq_s v_2'$, $v_2' \leq_s v_1$, and $v_2' \leq_s z_2$, and
let $v_3'$ now denote a vertex such that $w_3' \leq_s v_3'$, $w_1 \leq_s v_3'$, $v_3' \leq_s v_1$, and $v_3' \leq_s z_3$. (These vertices exist because $\sd'_{\cb_k}$ is bowtie-free under the $s$-order.) See Figure \ref{fig:flag3}(a).
Now in $\lk(w_2', \sd'_{\cb_k})$, $\{w_2,v_1',v_2'\}$ is pairwise $s$-upper bounded: $z_1$ is an upper bound of $\{w_2,v_1'\}$, $v_1$ is an upper bound of $\{v_1', v_2'\}$, and $z_2$ is an upper bound of $\{v_2', w_2\}$. So $\{w_2,v_1',v_2'\}$ has an $s$-upper bound $v_2$. Note that $w_1 \leq_s v_1' \leq_s v_2$ and $w_3' \leq_s v_2' \leq_s v_2$, so $v_2$ is an upper bound of $\{w_1,w_2,w_3'\}$. See Figure \ref{fig:flag3}(b).

\begin{figure}[!ht]
    \centering 

    \begin{multicols}{2}
        
    \includegraphics[width=0.75\linewidth]{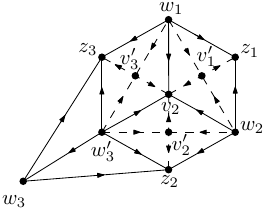}

    (a)
    
    \includegraphics[width=0.75\linewidth]{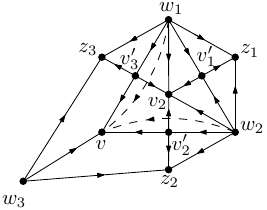}

    (b)
    
    \end{multicols}
    
    \caption{Finding $v$}
    \label{fig:flag4}
\end{figure}
 
We conclude by examining the vertices $\{w_1,w_2,w_3'\}$. If $s(w_3) = 1$, define $v = v_2$. In this case, $v = v_2$ is already an upper bound of $\{w_1,w_2,w_3\}$. Otherwise, define $v$ as follows. 
Let 
$\gamma_1^{(3)} = (w_1, v_2, w_2, z_1)$,
$\gamma_2^{(3)} = (w_2, v_2, w_3',z_2)$, and
$\gamma_3^{(3)} = (w_3',v_2, w_1, z_3)$.
To simplify notation again, 
let $v_1'$ now denote a vertex such that $w_1 \leq_s v_1'$, $w_2 \leq_s v_1'$, $v_1' \leq_s v_2$, and $v_1' \leq_s u_1$, 
let $v_2'$ now denote a vertex such that $w_2 \leq_s v_2'$, $w_3' \leq_s v_2'$, $v_2' \leq_s v_2$, and $v_2' \leq_s u_2$, 
and
let $v_3'$ now denote a vertex such that $w_3'\leq_s v_3'$, $w_1 \leq_s v_3'$, $v_3' \leq_s v_2$, and $v_3' \leq_s u_3$. (These all exist since $\sd_\cb'$ is bowtie free.) See Figure \ref{fig:flag4}(a).
Now in $\lk(w_3', \sd'_{\cb_k})$, $z_2$ is an ($s$-)upper bound of $\{w_3,v_2'\}$, $v_2$ is an upper bound of $\{v_2',v_3'\}$, and $z_3$ is an upper bound of $\{v_3',w_3\}$. So $\{w_3,v_2',v_3'\}$ has an $s$-upper bound $v$. Note that $w_1 \leq_s v_3' \leq_s v$ and $w_2 \leq_s v_2' \leq_s v$, so $v$ is an $s$-upper bound of $\{w_1,w_2,w_3\}$ (see Figure \ref{fig:flag4}(b)). 
\end{proof}

\let\c\ced
\bibliographystyle{alpha}
\bibliography{mybib}

@article{bjorner2003nerves,
  title={Nerves, fibers and homotopy groups},
  author={Bj{\"o}rner, Anders},
  journal={Journal of Combinatorial Theory, Series A},
  volume={102},
  number={1},
  pages={88--93},
  year={2003},
  publisher={Elsevier}
}

@article{huang2024cycles,
  title={Cycles in spherical {D}eligne complexes and application to ${K}(\pi, 1)$-conjecture for {A}rtin groups},
  author={Huang, Jingyin},
  journal={arXiv preprint arXiv:2405.12068},
  year={2024}
}

@article{huang2023labeled,
    AUTHOR = {Huang, Jingyin},
     TITLE = {Labeled four cycles and the {$K(\pi,1)$}-conjecture for
              {A}rtin groups},
   JOURNAL = {Invent. Math.},
  FJOURNAL = {Inventiones Mathematicae},
    VOLUME = {238},
      YEAR = {2024},
    NUMBER = {3},
     PAGES = {905--994},
      ISSN = {0020-9910,1432-1297},
   MRCLASS = {20F36},
MRREVIEWER = {Dawid\ Kielak},
       DOI = {10.1007/s00222-024-01294-w},
       URL = {https://doi-org.proxy.lib.ohio-state.edu/10.1007/s00222-024-01294-w},
}

@article{huang2024Dn,
	title={On spherical {D}eligne complexes of type ${D}_n$},
	author={Huang, Jingyin},
	journal={arXiv preprint arXiv:2405.11374},
	year={2024}
}

@article{huang2025353,
  title={353-combinatorial curvature and the $3$-dimensional ${K}(\pi,1)$ conjecture},
  author={Huang, Jingyin and Przytycki, Piotr},
  journal={arXiv preprint arXiv:2509.06914},
  year={2025}
}

@book{davey2002introduction,
  title={Introduction to lattices and order},
  author={Davey, Brian A and Priestley, Hilary A},
  year={2002},
  publisher={Cambridge university press}
}

@article{paris1993deligne,
  title={The Deligne complex of a real arrangement of hyperplanes},
  author={Paris, Luis},
  journal={Nagoya Mathematical Journal},
  volume={131},
  pages={39--65},
  year={1993},
  publisher={Cambridge University Press}
}

@article{falk1995k,
  title={${K} (\pi, 1)$ arrangements},
  author={Falk, Michael},
  journal={Topology},
  volume={34},
  number={1},
  pages={141--154},
  year={1995},
  publisher={Elsevier}
}

@article{terao1986modular,
  title={Modular elements of lattices and topological fibration},
  author={Terao, Hiroaki},
  journal={Advances in Mathematics},
  volume={62},
  pages={135--154},
  year={1986},
  publisher={Elsevier BV}
}

@article {CharneyDavis,
	AUTHOR = {Charney, Ruth and Davis, Michael W.},
	TITLE = {The {$K(\pi,1)$}-problem for hyperplane complements associated
	to infinite reflection groups},
	JOURNAL = {J. Amer. Math. Soc.},
	FJOURNAL = {Journal of the American Mathematical Society},
	VOLUME = {8},
	YEAR = {1995},
	NUMBER = {3},
	PAGES = {597--627},
	ISSN = {0894-0347},
	MRCLASS = {52B30 (20F55 55P20)},
	MRNUMBER = {1303028},
	MRREVIEWER = {Hiroaki Terao},
	DOI = {10.2307/2152924},
	URL = {http://dx.doi.org/10.2307/2152924},
}

@book{tits1974buildings,
  title={Buildings of spherical type and finite BN-pairs},
  author={Tits, Jacques},
  year={1974},
  publisher={Springer}
}

@article{brady2010braids,
  title={Braids, posets and orthoschemes},
  author={Brady, Tom and McCammond, Jon},
  journal={Algebraic \& Geometric Topology},
  volume={10},
  number={4},
  pages={2277--2314},
  year={2010},
  publisher={Mathematical Sciences Publishers}
}

@article{godelle2012k,
	title={${K}(\pi, 1)$ and word problems for infinite type Artin--Tits groups, and applications to virtual braid groups},
	author={Godelle, Eddy and Paris, Luis},
	journal={Mathematische Zeitschrift},
	volume={272},
	number={3},
	pages={1339--1364},
	year={2012},
	publisher={Springer}
}

@article{s87,
	title={Topology of the complement of real hyperplanes in $\mathbb{C}^{N}$},
	author={Salvetti, Mario},
	journal={Invent. math},
	volume={88},
	number={3},
	pages={603--618},
	year={1987}
}

@book{lek,
	title={The homotopy type of complex hyperplane complements},
	author={van der Lek, Harm},
	publisher = {PhD thesis, Nijmegan},
	year={1983}
}

@book {BridsonHaefliger1999,
	AUTHOR = {Bridson, Martin R. and Haefliger, Andr\'e},
	TITLE = {Metric spaces of non-positive curvature},
	SERIES = {Grundlehren der Mathematischen Wissenschaften [Fundamental
	Principles of Mathematical Sciences]},
	VOLUME = {319},
	PUBLISHER = {Springer-Verlag, Berlin},
	YEAR = {1999},
	PAGES = {xxii+643},
	ISBN = {3-540-64324-9},
	MRCLASS = {53C23 (20F65 53C70 57M07)},
	MRNUMBER = {1744486},
	MRREVIEWER = {Athanase Papadopoulos},
	DOI = {10.1007/978-3-662-12494-9},
	URL = {http://dx.doi.org/10.1007/978-3-662-12494-9},
}

@article {deligne,
	AUTHOR = {Deligne, Pierre},
	TITLE = {Les immeubles des groupes de tresses g\'en\'eralis\'es},
	JOURNAL = {Invent. Math.},
	FJOURNAL = {Inventiones Mathematicae},
	VOLUME = {17},
	YEAR = {1972},
	PAGES = {273--302},
	ISSN = {0020-9910},
	MRCLASS = {32C40 (55A25)},
	MRNUMBER = {0422673},
	MRREVIEWER = {Alan H. Durfee},
	DOI = {10.1007/BF01406236},
	URL = {http://dx.doi.org/10.1007/BF01406236},
}

@article{hirai2020uniform,
  title={Uniform modular lattices and affine buildings},
  author={Hirai, Hiroshi},
  journal={Advances in Geometry},
  volume={20},
  number={3},
  pages={375--390},
  year={2020},
  publisher={De Gruyter}
}

@book {dbook,
	AUTHOR = {Davis, Michael W.},
	TITLE = {The geometry and topology of {C}oxeter groups},
	SERIES = {London Mathematical Society Monographs Series},
	VOLUME = {32},
	PUBLISHER = {Princeton University Press, Princeton, NJ},
	YEAR = {2008},
	PAGES = {xvi+584},
	ISBN = {978-0-691-13138-2; 0-691-13138-4},
	MRCLASS = {20F55 (05B45 05C25 51-02 57M07)},
	MRNUMBER = {2360474},
	MRREVIEWER = {Ralf Gramlich},
}

@article{HogeRoehrle2014,
  author  = {Hoge, Torsten and R{\"o}hrle, Gerhard},
  title   = {Supersolvable reflection arrangements},
  journal = {Proceedings of the American Mathematical Society},
  volume  = {142},
  number  = {11},
  year    = {2014},
  pages   = {3787--3799},
  doi     = {10.1090/S0002-9939-2014-12144-7}
}

@article{JambuTerao1984,
  author  = {Jambu, Michel and Terao, Hiroaki},
  title   = {Free arrangements of hyperplanes and supersolvable lattices},
  journal = {Advances in Mathematics},
  volume  = {52},
  number  = {3},
  year    = {1984},
  pages   = {248--258},
  doi     = {10.1016/0001-8708(84)90024-0}
}

@article{AmendHogeRoehrle2014,
  author  = {Amend, Nils and Hoge, Torsten and R{\"o}hrle, Gerhard},
  title   = {Supersolvable restrictions of reflection arrangements},
  journal = {Journal of Combinatorial Theory, Series A},
  volume  = {127},
  year    = {2014},
  pages   = {336--352},
  doi     = {10.1016/j.jcta.2014.07.002}
}

@article{Stanley1971,
  author  = {Stanley, Richard P.},
  title   = {Modular elements of geometric lattices},
  journal = {Algebra Universalis},
  volume  = {1},
  year    = {1971},
  pages   = {214--217},
  doi     = {10.1007/BF02944982}
}

@article{huang2021helly,
	title={Helly meets Garside and Artin},
	author={Huang, Jingyin and Osajda, Damian},
	journal={Inventiones mathematicae},
	volume={225},
	pages={395--426},
	year={2021},
	publisher={Springer}
}

@inproceedings{paris2014k,
	title={${K}(\pi,1)$ conjecture for {A}rtin groups},
	author={Paris, Luis},
	booktitle={Annales de la Facult{\'e} des sciences de Toulouse: Math{\'e}matiques},
	volume={23(2)},
	pages={361--415},
	year={2014}
}

@article{cumplido2020parabolic,
    AUTHOR = {Cumplido, Mar\'ia and Martin, Alexandre and Vaskou, Nicolas},
     TITLE = {Parabolic subgroups of large-type {A}rtin groups},
   JOURNAL = {Math. Proc. Cambridge Philos. Soc.},
  FJOURNAL = {Mathematical Proceedings of the Cambridge Philosophical
              Society},
    VOLUME = {174},
      YEAR = {2023},
    NUMBER = {2},
     PAGES = {393--414},
      ISSN = {0305-0041,1469-8064},
   MRCLASS = {20F65 (20F36)},
  MRNUMBER = {4545212},
MRREVIEWER = {Michael\ Dougherty},
       DOI = {10.1017/S0305004122000342},
       URL = {https://doi-org.proxy.lib.ohio-state.edu/10.1017/S0305004122000342},
}

@article{haettel2021lattices,
    AUTHOR = {Haettel, Thomas},
     TITLE = {Lattices, injective metrics and the {$K(\pi,1)$} conjecture},
   JOURNAL = {Algebr. Geom. Topol.},
  FJOURNAL = {Algebraic \& Geometric Topology},
    VOLUME = {24},
      YEAR = {2024},
    NUMBER = {7},
     PAGES = {4007--4060},
      ISSN = {1472-2747,1472-2739},
   MRCLASS = {20E42 (05B35 06A12 20F36 52A35)},
MRREVIEWER = {Mar\'ia\ Cumplido},
       DOI = {10.2140/agt.2024.24.4007},
       URL = {https://doi-org.proxy.lib.ohio-state.edu/10.2140/agt.2024.24.4007},
}

@article{lang2013injective,
	title={Injective hulls of certain discrete metric spaces and groups},
	author={Lang, Urs},
	journal={Journal of Topology and Analysis},
	volume={5},
	number={03},
	pages={297--331},
	year={2013},
	publisher={World Scientific}
}

@article{falk1986homotopy,
  title={On the homotopy theory of arrangements},
  author={Falk, Michael and Randell, Richard},
  journal={Complex analytic singularities},
  volume={8},
  pages={101--124},
  year={1986},
  publisher={North-Holland Amsterdam}
}

@article{falk1998homotopy,
  title={On the homotopy theory of arrangements, II},
  author={Falk, Michael and Randell, Richard},
  journal={Arrangements in Tokyo},
  pages={185--215},
  year={1998}
}

@article{paolini2021proof,
	title={Proof of the ${K} (\pi, 1)$ conjecture for affine {A}rtin groups},
	author={Paolini, Giovanni and Salvetti, Mario},
	journal={Inventiones mathematicae},
	volume={224},
	number={2},
	pages={487--572},
	year={2021},
	publisher={Springer}
}

\end{document}